\documentclass[12pt]{amsart}
\usepackage{latexsym,fancyhdr,amssymb,color,amsmath,amsthm,graphicx,listings}
\newtheorem{thm}{Theorem}        
\newtheorem{lemma}{Lemma}           
\def\obs#1{ \vspace{2mm} \parbox{15.2cm}{{\bf Observation:} #1} }
\def\quest#1{ \vspace{2mm} \parbox{15.2cm}{{\bf Question:} #1} }
\usepackage[section]{placeins}  
\let\paragraph\subsection    \newcommand{\E}{{\rm E}}
\title{A Dehn type quantity for Riemannian manifolds} 
\author{Oliver Knill} 
\date{02/31/2020}
\address{Department of Mathematics \\ Harvard University \\ Cambridge, MA, 02138 }

\begin{document}
\begin{abstract}
We look at the functional $\gamma(M) = \int_M K(x) dV(x)$ for 
compact Riemannian 2d-manifolds $M$, where 
$K(x) = (2d)! (d!)^{-1} (4\pi)^{-d} \int_T \prod_{k=1}^d K_{t_{2k},t_{2k+1}}(x) dt$
involves products of $d$ sectional curvatures $K_{ij}(x)$ 
averaged over the space $T \sim O(2d)$ of all orthonormal frames $t=(t_1,\dots,t_{2d})$. 
A discrete version $\gamma_d(M)$ with 
$K_d(x) = (d!)^{-1} (4\pi)^{-d} \sum_\sigma \prod_{k=1}^d K_{\sigma(2k-1),\sigma(2k)}$
sums over all permutations $\sigma$ of $\{1,\dots,2d \}$.
Unlike Euler characteristic which by Gauss-Bonnet-Chern is $\int_M K_{GBC} \; dV=\chi(M)$,
the quantities $\gamma$ or $\gamma_d$ are in general metric dependent.
We are interested in $\delta(M)=\gamma(M)-\chi(M)$ because if $M$ 
has curvature sign $e$, then $\gamma(M) e^d$ and $\gamma_d(M)$ are positive while $\chi(M) e^d>0$ 
is only conjectured. We compute $\gamma_d$ in a few concrete examples like 2d-spheres, 
the $4$-manifold $\mathbb{CP}^2$, the $6$ manifold $SO(4)$ or the $8$-manifold $SU(3)$. 
\end{abstract}

\maketitle

\section{The functional}

\paragraph{}
When studying manifolds with definite curvature sign, one is naturally led to quantities like
$$ \int_{M} \frac{1}{d!(4\pi)^d} \sum_{\pi \in S_{2d}} \prod_{k=1}^d K_{\pi(2k-1) \pi(2k)}(x) \; dV(x) $$
on compact Riemannian $2d$-manifolds $M$ involves curvatures of pairwise 
perpendicular coordinate $2$-planes in the tangent space $T_xM$. The definition makes use of an
orthonormal coordinate system in $T_xM$ and depends in general on it, like on how $M$ is parametrized.
Averaging over a one parameter rotation in $T_xM$ makes it equal to the Haar average over the 
Stiefel manifold $T \sim O(2d)$ of orthonormal frames and so coordinate independent. 
We first thought that $\gamma(M,g)= \int_M K(x) \; dV(x)$ with
$$  K(x) = (2d)! (d!)^{-1} (4\pi)^{-d} \int_T \prod_{k=1}^d K_{t_{2k},t_{2k+1}}(x) dt $$
is independent of the Riemannian metric $g$ chosen on $M$ and that 
the question whether $\gamma(M)=\chi(M)$ depends on whether $M$ can be partitioned properly, hence the name
``Dehn-Type quantity" in the title. As $\gamma(M,g)$ can depend on the metric and the frame, 
it is a quantity which like the Hilbert action could be studied as a variational problem.

\paragraph{}
In most examples we actually computed $\gamma(M)$ as a finite sum $\gamma_d(M)$ and not the
true coordinate independent average as finding $\gamma(M)$ requires an additional costly integration
at every point $x \in M$. We are interested to estimate $\delta(M)=\gamma(M)-\chi(M)$ or $\delta_d(M)=\gamma_d(M)-\chi(M)$. 
The curvature $K(x)$ can be seen locally as an index
expectation of Poincar\'e-Hopf indices of probability spaces of Morse functions in which a product measure
tries to emulate a product of 2-manifolds. As this interpretation only works locally in small regions,
we have to cut up $M$ into polyhedra $M_j$ and understand what happens
with boundary curvature when gluing the pieces $M_j$ together. Much of the content of this effort was trying
to control this boundary curvature. Unfortunately, doing that will need much more work. 
The Allendoerfer-Weil gluing works perfectly fine in the Gauss-Bonnet-Chern case, leading
to $\chi(M)$ but there is some gluing curvature in the $\gamma(M)$ case. It might be possible that
$\delta(M)=0$ is linked to symmetry on $(M,d)$ as this requires less cutting is needed.
It would not surprise to be the case in light of \cite{Kennard2013,AmannKennard}. We comment on this
at the very end. 

\paragraph{}
If $M$ can be partitioned into contractible Riemannian polyhedra $M_j$ which at every boundary point satisfy 
a self-dual Fenchel cone condition, we call such polyhedra {\bf orthotope}. It allows us to decompose $M$ into 
polytop pieces $M_j$ on which we chose curvature as index expectation using probability spaces $\Omega_j$ 
of Morse functions on $M_j$. Unlike in the classical Gauss-Bonnet-Chern theorem, where a single probability space works, the 
probability spaces $\Omega_j$ in different cells can now be utterly unrelated. In a polytope case $M=N \times [a,b]$
where $N$ can have a boundary and the metric on $N$ is independent of the second variable, then $\gamma(M)=\chi(N)$.

\paragraph{}
We initially hoped to estimate the difference $\delta(M)=\gamma(M)-\chi(M)$ by estimating the glue curvature for some 
partition or in the limit when the diameters of the $M_j$ go to zero. One can think that if the fluctuations of $g$ or of
curvature are small and the overall curvature are positive enough, then $\gamma(M)-\chi(M)$ is so small that it forces also
$\chi(M)>0$, which is the actual goal to show for positive curvature
manifolds. An isometric circle action can not establish $\gamma(M)=\chi(M)$ in general, as $SU(3)$ shows which has 
a $\mathbb{T}^2$ action.

\paragraph{}
The fact that the gluing produces non-removable curvature at the intersections leads to an
abstract combinatorial problem for a geometric realization $M=\bigcup M_k$ of a finite abstract simplicial
complex $G$ with maximal simplices generating the chambers $M_k$ in $M$. If $f$ is a Morse function on $M$
which defines Morse functions on each $M_k$ in the sense defined below and $x$ is a simplex in $G$, then we can
look at $h(x) = \sum_{m \in x} i_f(m) = 1$ with Morse indices $i_f(m) \in \{-1,1\}$. 
This leads to the counting matrix $L(x,y)$ counting the number of simplices
in $x \cap y$. This is a positive definite quadratic form in $SL(n,\mathbb{Z})$ of symplectic nature as $L$ is 
conjugated to its inverse $g$, the Green function which is again integer valued and for which $\sum_{x,y} g(x,y) = |G|$
is the number of simplices in $G$. For a global Morse function $f$, we can count the indices using 
exclusion-inclusion. This leads to $h(x)=\omega(x)=(-1)^{{\rm dim}(x)}$ and $\sum_x h(x) = \chi(G)=\chi(M)$ is the 
Euler characteristic of $G$ and equal to the Euler characteristic of $M$. Now, if $f_k$ are Morse on $M_k$ and
independent of each other, then the gluing does not work and $h(x)$ can be an arbitrary integer-valued function on 
the simplicial complex. There is still a connection matrix $L$ which satisfies ${\rm det}(L) = \prod_x h(x)$ 
and $\sum_{x,y} g(x,y) = \sum_x h(x)$. The situation of having no invariant but a metric-dependent quantity brings
the functional $\gamma(M)$ closer to the concept of {\bf energized simplicial complexes} \cite{EnergizedSimplicialComplexes},
where the topological $h(x)=\omega(x)$ leading to Euler characteristic $\chi(M)$ is replaced by a more general $h(x)$ for which 
$\gamma(M) = \sum_x h(x)$ differs from $\chi(M)$. Averaging over probability spaces of Morse functions then 
renders $h(x)$ real-valued for which the abstract combinatorial results still work. 

\paragraph{}
The constant $C_d=(d!)^{-1} (4\pi)^{-d}$ 
in the definition of $\gamma$ is chosen so that $\gamma(M)=\chi(M)$ for spheres or products of $2$-manifolds.
For the $2d$-dimensional sphere $\mathbb{S}^{2d}$, where the sectional curvatures
$K_{ij}=1$, we have $|\mathbb{S}^{2d}|= 2d!(4\pi)^d/(2d)!$ and $C_d \cdot |\mathbb{S}^{2d}| \cdot (2d)! = 2$. 
In the case $d=1$, have $\gamma(M)=\chi(M)$. For a unit sphere $M=\mathbb{S}^d$ the sectional curvature are $K_{ij}=1$
so that $K=2/|\mathbb{S}^{2d}|$ agrees with the {\bf Gauss-Bonnet-Chern curvature} $K_{GBC}$. 
Integrating over $M$ immediately gives $2$ as the volume $|\mathbb{S}^{2d}|$ of the $2d$-sphere is $2d! (4 \pi)^d/(2d)!$.
For the product manifold $M=\mathbb{S}^2 \times \mathbb{S}^2$, the volume is $6$ times bigger and the
curvature is $3$ times smaller because there are $8$ terms which are $1$ and $16$ terms which are $0$.
The Euler characteristic is therefore two times larger. Indeed, it is 
$\chi(\mathbb{S}^2 \times \mathbb{S}^2) = 2^2 = 4$.

\paragraph{} 
The following observations are all immediate. The first one is a property shared with 
Euler characteristic. It is important however that the metric on the product is the product
metric, meaning that distances in $M_1$ do not depend on coordinates in $M_2$ and vice versa:

\obs{
If $M=M_1 \times M_2$ with the product metric and compatible frame, then
$\gamma_d(M) = \gamma(M_1) \times \gamma(M_2)$. 
}

This follows from the fact that any 2-plane which is mixed has zero curvature.
However, the result $\gamma_d(M)$ can depend on the frame. This already happens for $M=S^2 \times S^2$. 

\begin{figure}[!htpb]
\scalebox{0.4}{\includegraphics{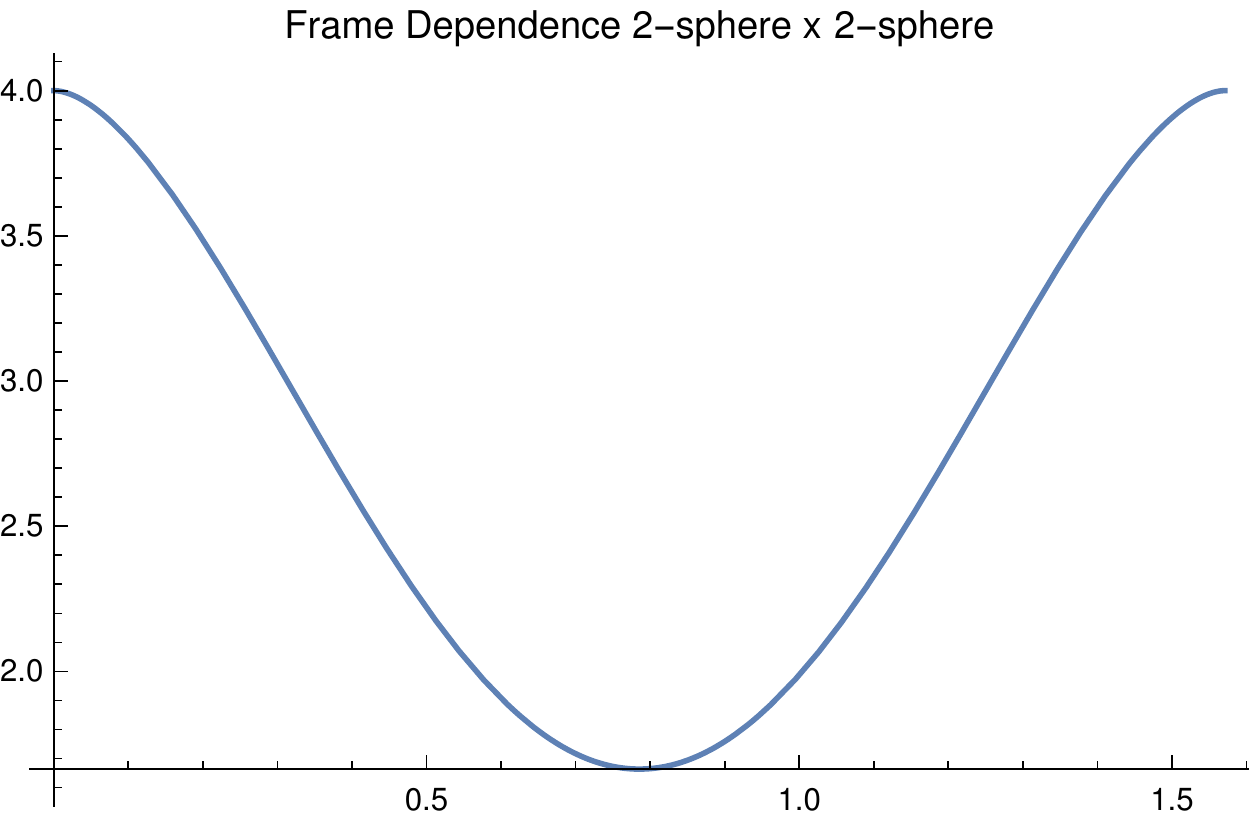}}
\scalebox{0.4}{\includegraphics{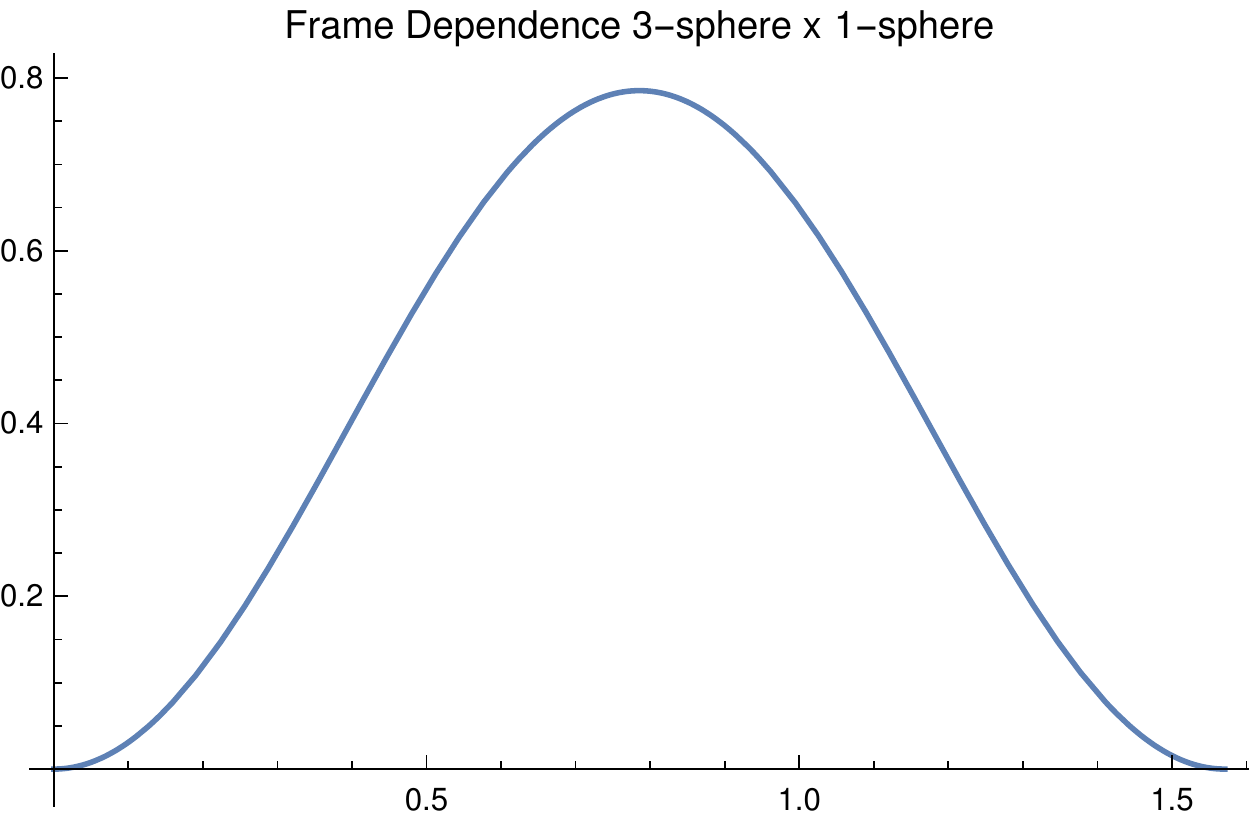}}
\label{sphere}
\caption{
The value $\gamma_d(M)$ can depend on the frame even in product situations.
We see two examples of 4-manifolds, the left shows $\mathbb{S}^2 \times \mathbb{S}^2$
(for which one does not know whether some positive metric exists on), the right shows
$\mathbb{S}^3 \times \mathbb{S}^1$. In both cases, we rotate around an axes
by an angle in $[0,\pi/2]$. 
}
\end{figure}

\paragraph{}
We have not defined $\gamma$ a priory for odd-dimensional manifolds but in dimension $2d+1$, the
same definition can be used, we just have in each product just $d$ perpendicular planes. 

\paragraph{}
Especially: 

\obs{
In general, for any products of $2$-manifolds $M=M_1 \times M_2 \times \cdots \times M_d$
with product metric and orthogonal frame compatible with the product we have $\chi(M) = \gamma_d(M)$. 
}

This follows from the fact that for any mixed plane visiting two different manifolds, the curvature is zero. 
The only contributions which remain are $K_1(x_1) \cdots K_d(x_d)$, where $K_k$ are the Gaussian curvatures
in the factors $M_j$. So, $\gamma(M) = \prod_k \chi(M_k)$. Again, also here, if the frame is turned in $M$ then 
$\gamma$ can become smaller. The frame for which $\gamma(M)$ is maximal is the frame which contains orthogonal
sub-frames spanning each factor $M_k$. 

\paragraph{}
Very similar is the case when one factor is odd-dimensional. We have to state that separately because we don't 
have necessarily $\gamma(N)=0$ for odd-dimensional $N$: 

\obs{
In general, for any products $M=N \times T$ with odd-dimensional $N,T$ with the product metric and
one basis vector in $T$, then $\chi(M) = \gamma(M)=0$.
}

This follows from the fact that one of the factors in the product of sectional curvatures has to come from a 
mixed plane and so has zero curvature. This means for example that for any {\bf space-time metric} $N \times \mathbb{T}^1$
(where the circle $\mathbb{T}^1$ plays the role of time) and for which the metric on $N$ is time independent, 
then $\gamma(M)=\chi(M)=0$. An example of $\mathbb{T}^3 \times \mathbb{T}$ with time dependent metric on$\mathbb{T}^3$ is
in the example section. Already for $\mathbb{S}^3 \times \mathbb{T}^1$ there are coordinate frames for which 
$\gamma_d(M)$ is non-zero. 

\paragraph{}
For $d=1$ and $M=\mathbb{S}^2$ with radius $1$ in particular, $K$  
$2 |\mathbb{S}^2|^{-1} = (4 \pi)^{-1}$ 
times the Gaussian curvature $1$ of the sphere. Already for $d \geq 2$, the curvature $K$ 
differs from the {\bf Gauss-Bonnet-Chern integrand}
\cite{HopfCurvaturaIntegra,Allendoerfer,Fenchel,AllendoerferWeil,Chern44,Chern1990,Cycon}
$$ K_{GBC} = 2^{-d} C_d \sum_{\sigma,\pi} {\rm sign}(\pi) {\rm sign}(\sigma) 
             R_{\pi_1,\pi_2,\sigma_2,\sigma_1} \cdots R_{\pi_{2d-1},\pi_{2d},\sigma_{2d-1},\sigma_{2d}} \; , $$
which involves a summation over all pairs of permutations $\pi,\sigma$ of $\{1, \cdots, 2d\}$ and
where the Riemannian curvature tensor expressions $R_{ijkl}(x)$ also refer to an orthonormal coordinate system at $x$.
The curvature $K$ sums over the single set of all permutations of $\{1, \cdots, 2d\}$:
$$ K(x) =   C_d \sum_{\sigma}   R_{\sigma_1,\sigma_2,\sigma_1,\sigma_2}(x) \cdots R_{\sigma_{2d-1},\sigma_{2d},\sigma_{2d-1},\sigma_{2d}}(x) \;  $$
and where also at each point $x$, the curvature tensor entry $R_{ijkl}(x)$ 
is computed using normal coordinates at $x$. Compare $K$ also
with the {\bf scalar curvature} which is $2K$ for $d=1$ and in normal coordinates given by 
$K_{Scal}(x) = \sum_{i,j} R_{ijij}$ and integrating $\int_M K_{Scal}(x)\; dV$ to the {\bf Hilbert action}.

\paragraph{}
The curvature $K_{KBC}$ also explores off-diagonal terms of the Riemannian curvature tensor $R$. 
For example, for the non-negative curvature $6$-manifold $M=SO(4)$ equipped with the
bi-invariant metric obtained from the Killing form, where the curvature tensor is evaluated on 
vector fields $X,Y,U,V$ as $R(X,Y,U,V)=g([X,Y],[U,V])$, there are $6!^2=720^2$ summands for $K_{GBC}$, while
only $720$ appear for $K$. One can see immediately why $K$ is constant zero. While $K_{GBC}$ is not zero everywhere,
it not so obviously sums up to zero, the Euler characteristic $\chi(M)$ of the compact Lie group $M=SO(4)$. 
In this example however, the fact that the universal cover $Spin(4)$ is the product $S^3 \times S^3$ of two spheres
is special. 

\paragraph{}
We guessed first that $\gamma(M)=\chi(M)$ for all Riemannian $2d$-manifolds $M$ and even were under the impression
that $\gamma$ is independent of the metric and that $\gamma_d(M)=\gamma(M)$. 
But then we computed $\gamma$ for the $8$-manifold 
$SU(3)$, where $\gamma(M) > \chi(M)=0$ and thought of the inequality $\gamma(M) \geq  \chi(M)$ which 
is not true: Cliff Taubes then sent us an example with a family of metrics on $M=\mathbb{T}^4$ in which 
$\gamma(M)$ depends on the metric and can become both positive or negative and also that
$\delta(M)=\gamma(M)-\chi(M)$ can take any value $c$: for every $c$ there is a metric on 
the 4-torus $M$ such that $\delta(M)=c$. Of course, $\delta(M)=\delta_d(M)=0$ in the independent
product case $(\mathbb{T}^2,g_1) \times (\mathbb{T}^2,g_2)$.

\paragraph{}
Jeffrey Chase pointed out to us a theorem of Gilkey \cite{Gilkey1975} who proved a conjecture of Singer 
stating that the only {\bf diffeomorphism invariants} among compact Riemannian manifolds which are obtained by
integrating local formulas in $R$ derivatives of $R$ and $g$ are multiples of the Euler 
characteristic and that the curvature must be a multiple of $K_{GBC}(x)$ plus a divergence
of a vector field. This theorem tells that it is futile to look for integrands similar
to $K_{GBC}(x)$ which are metric invariant. The topic has been spun further by weakening
the symmetry and look for Riemannian invariants which are only invariant under conformal scaling
modifications of the metric $g$. A conjecture of Deser-Schwimmmer from 1993 asserted that such a curvature
has to be a multiple of $K_{GBC}$ plus a divergence plus a conformal invariant $W$. This
was proven later by Alexakis \cite{Alexakis1,Alexakis2}. 

\paragraph{} 
The just mentioned work on invariant theory also shows that most
{\bf index expectation curvatures} are not given in terms of analytic expressions in $g,R,\nabla R$.
Curvatures $K(x) = {\rm E}[i_f(x)]$ are expectations of Poincar\'e-Hopf indices of Morse 
functions on $M$, integrate up to $\chi(M)$ and are by nature local too in the sense that the curvature computed in 
the ball $B_r(x)$ is independent of $r$ as long as $r>0$. These curvatures are in general 
also not coordinate independent. The invariance theory started with Herman Weyl.

\paragraph{}
The product curvature $K(x)=\prod_k K_k(x_k)$ in a product 
$M_1 \times \cdots \times M_d$ is an example of an index expectation curvature that is not
isotropic. We symmetrize it by summing over all permutations of coordinate axes and average
over all rotations to make it isotropic. Jason de Vito pointed out to me the coordinate dependence
by giving frame-dependent values in examples like $\mathbb{CP}^2$. It actually does in most
cases. The frame-dependence is more serious for physics. The metric dependence leads to the interpretation
of depending on some mass distributed on the manifold. $\gamma$ 
must be seen as a higher order version of the Hilbert action which
explores curvature correlations between d coordinate planes in a $2d$-manifold. As it 
depends on the metric, one can look at the question which metrics maximize or minimize it. 
Because it also depends on the frame, we could use it to find a frame which extremizes curvature. 

\paragraph{}
The quantity $\gamma$ or $\gamma_d$ is motivated by the {\bf Hopf conjectures} from 1931 \cite{Hopf1932} 
which have reappeared in the sixties \cite{BishopGoldberg} and \cite{Chern1966} and are listed 
as problems 8) and 10) in the problem collection \cite{YauSeminar1982}. 
The idea had been to look for a curvature which only involves sectional curvatures and which is positive if
sectional curvature is positive.  The functional appears to be natural as it is a {\bf $d$-point 
correlation function} of sectional curvature on $2d$-manifolds. 
In this picture, scalar curvature is a {\bf $1$-point correlation function}. We only know for now:

\obs{
For manifolds with $\delta(M)=0$ or $\delta_d(M)=0$ (for some choice of frame), the Hopf product conjecture holds.
It might be enough to compute $\gamma_d(M)$ and estimate $\delta_d(M)$ to establish $\chi(M) \neq 0$. 
}

\paragraph{}
For arbitrary smooth compact manifolds, one can look at $k$-point correlation functions which looks at all possible 
$k$-tuples of orthogonal $2$-planes, multiplies those curvatures, then integrates over the manifold. Also 
Aside for $k=1$, correlation functionals $\gamma_k(M)$ (or discrete versions $\gamma_{k,d}$ of it which in general
disagree for $k>1$) seem not have been investigated yet except for the
{\bf Hilbert action function} $\gamma_1(M)$.
The functionals $\gamma_k(M)$ or $\gamma_{k,d}(M)$ are all positive for positive curvature manifolds.
Computing $\gamma(M)$ (and even $\gamma_d(M)$) numerically is already challenging for $4$-manifolds, the
reason being that we have to produce orthonormal coordinate systems $t$ for $T_xM$ at each point $x$ of $M$.
It would be interesting to investigate this even in special cases like for K\"ahler manifolds as we
even do not know yet $\gamma(M)$ for the K3 surface, a smooth 4-manifold with $\chi(M)=24$. 
For complex manifolds, one could modify $\gamma$ also by using holomorphic sectional curvatures.

\section{Coordinate dependence}

\paragraph{}
Expressions close to curvature $K(x)$ have been explored as expectation of discrete 
Poincar\'e-Hopf indices \cite{DiscreteHopf,DiscreteHopf2}. This led us to $\gamma_d(M)$. It is based on integral geometry 
\cite{Banchoff67,Banchoff1970}, articles which pioneered such integral geometric approaches to Riemannian geometry.
It is a general principle that integrating the Poincar\'e-Hopf relation $\chi(M) = \sum_{x \in M} i_f(x)$
over a probability space $(\Omega,\mathcal{A},P)$ of Morse functions $f$ gives a curvature $K(x) = {\rm E}[i_f(x)]$
satisfying Gauss-Bonnet. If one uses a probability space so that in each $2$-plane $t_k$, an independent set of
Morse functions $f_k$ is used and $f$ is the sum of these functions $f_k$. The indices of $f$ are then the product
of the indices of $f_k$. By independence, this gives $K(x,t)=\E[i_{f_1,\dots,f_d}] =  \E[i_{t_1,f_1} \cdots i_{t_d,f_d}]$
$= \E[i_{{t_1},f_1}] \cdots \E[i_{{t_d},f_d}]$ $= K_{t_1}(x) \cdots K_{t_d}(x)$ which is a
product of sectional curvatures evaluated at perpendicular t-coordinate planes at $x$.

\paragraph{}
The just mentioned curvature $K(x,t)$ depends on the coordinates $t$ but averaging point-wise over the fiber $T=O(M)_x$
produces a coordinate-independent $K(x)$. Integration over $T$ is natural as $T$ is a
Stiefel manifold of all {\bf orthonormal frames} at a point $x \in M$ and homeomorphic to the orthogonal group $O(2d)$ and
so carries a unique Haar probability measure $dt$. The push-forward curvature $K(x,t)$
coming from the frame bundle projection $\pi: O(M) \to M$ gives then the curvature $K(x)$. 
Numerically, we need at every point to find an orthonormal frame, which needs in general
a Gramm-Schmidt process, even if $M$ is given by an explicit parameterization. 

\paragraph{}
At any given point $x \in M$, we can average the sectional curvature product expression either as 
an integral over all frames leading to $K(x)$ or then as a finite sum 
of  permutations $\pi$ of $\{1,2,\dots,2d\}$. The finite sum 
$K_d(x) = K(x,t_0)=\sum_\pi \prod_{k=1}^d K_{\pi(2k-1),\pi(2k)}(x)/(2d!)$ with fixed
orthonormal frame $t_0$ is easier to compute than 
$K(x) = \int_{O(2d)} K(x,t) \; dt$ with Haar measure $dt$ on the orthogonal group $O(2d)$. 

\paragraph{}
The discrete curvature and averaged curvature agree at a point $x$ if some symmetry 
is present. This applies in product situations but only if the coordinate system matches the decomposition

\obs{
If $M=N \times T$, where $T$ is one-dimensional, then $K_d$ is zero if one of the basis vectors of the frame 
is in $T$.
}

In such an {\bf aligned case}, $K$ is the product of the $d$ sectional curvatures 
Already for a 4-manifold like $M=\mathbb{S}^3 \times \mathbb{S}^1$, the value $\gamma_d(M)$ can depend on the
coordinate frame. We can for example rotate in the $x_1,x_4$-plane and see $\gamma_d(M)$ change. 

\paragraph{}
Averaging principles equating continuum averages with discrete sums are known 
for {\bf Ricci curvature}, a finite sum of sectional curvatures
is equivalent to an average of sectional curvatures or that {\bf scalar curvature}, the trace of Ricci curvature
is an integral average of sectional curvatures. Ricci and scalar curvature are
given finite sum of sectional curvatures \cite{Petersen}.
The fact that in the case of Ricci and scalar curvature the discrete and continuum sum 
agree has been pointed out \cite{MorganRiemannian} and are a consequence of the Theorema egregia.
We first thought that this goes over to $K(x)$ but what happens is that if 
we rotate in a coordinate 2-plane $\Sigma$, then the other 2-planes intersecting $\Sigma$ also 
rotate around and so can change curvature. This can happen already in very symmetric situations 
like the complex projective plane $\mathbb{CP}^2$ as pointed out by Jason de Vito.

\paragraph{}
Here is a possible situation which on has not to average over the entire group $O(2d)$: 
If an arrow of time $v(x)$ exists such that that the sectional curvature 
$K(v(x),w)$ is independent of a space vector $w$ perpendicular to $v(x)$, then the discrete curvature sum
in a frame containing $v(x)$ as a vector is the total average.
The condition assures that rotating around $v(x)$ does not change the value. Now, this rotation subgroup
as well as all conjugations of it fix the value. If all these rotations generate the entire group $O(2d)$
the result is frame independent.  In physics, one would consider $v(x)$ a choice of a time direction
and the condition that $K(v(x),w)$ is independent of $w$ as some sort of local space isotropy. 

\paragraph{}
We know that if $\gamma(M)=0$, then this metric does not have strictly positive curvature. An example for
which $\gamma(M)=0$ for the standard Killing metric and a standard basis compatible with the splitting
of ${\rm Spin}(4)$ (a manifold mentioned at the end of the paper \cite{Bourguignon75})
for which it is not known whether it admits a metric of positive curvature or not.
The simply connected ${\rm Spin}(4) = \mathbb{S}^3 \times \mathbb{S}^3=SU(2) \times SU(2)$ 
is the universal cover of the $6$-manifold $SO(4)$ and has sectional curvatures with respect
to the bi-invariant metric which are explicitly given as $K(V,W)=|[{\rm ad}_V,{\rm ad}_W]|^2 \geq 0$. It is
special as it is only semi-simple. 
In \cite{YauSeminar1982} the question was raised to get a simply connected manifold of non-negative curvature
which does not admit a metric of positive curvature. This question generalizes the
product question of Hopf whether $\mathbb{S}^2 \times \mathbb{S}^2$ 
(a non-negative curvature manifold) admits a curvature which is positive. Both the Hopf question
about $\mathbb{S}^2 \times \mathbb{S}^2$ and the Yau question asking for examples of non-negative curvature
manifolds admitting no positive curvature are open. 

\section{When do we have equality?}

\paragraph{}
In an earlier version of this paper, we thought that a partition of the manifold 
into orthotope pieces $M_j$ would allow us to see that the curvature on the boundary 
can be discarded when gluing things together. This turned out to be wrong. When 
looking at a Riemannian polyhedron $M_k$ containing a point $x$ and making the curvature 
$K(x)$ isotropic at this point does not make it isotropic everywhere. While the 
discrepancy from isotropy goes to zero if the manifold is chopped up finer and finer,
there are also errors which manifest at the
boundary. These errors add up but with a fine partition, also the number of boundary
vertices grows too. This can lead to $\gamma(M) \neq \chi(M)$ in general. In each
polyhedron $M_j$ we have a match $\gamma(M_k) = \int_{M_k} E[i_f(x)] \; dV(x) = \chi(M_k)$. 
But $K(x)$ is only at the point equal to $E[i_f(x)]$.  	

\paragraph{}
Despite the fact that orthotop manifolds might eventually not matter in this
context, we leave here the original analysis about this concept. It actually leads to an interesting 
{\bf global geometric problem} in Riemannian geometry which seems not have been 
studied so far: 

\quest{
Can we chop up any compact Riemannian manifold into orthotop pieces? 
}

\paragraph{}
We say that a contractible Riemannian polyhedron $M$ is in the class
$\mathcal{O}$ of {\bf orthotope manifolds} if at every boundary point $x$ of $M$,
the {\bf Fenchel cone} $C$ generated by $B_r(x) \cap M$ in the limit $r \to 0$ is either
an Euclidean half space or equal to
the {\bf dual Fenchel cone} $\hat{C}$ at that point. We say that a Riemannian manifold
$M$ is in the class $\mathcal{O}$ if it is possible to see $M$ as a union of arbitrary 
small Riemannian polyhedra $M_j$ in the class or orthotope polyhedra 
$\mathcal{O}$ such that intersections $M_i \cap M_j$ are empty or in $\mathcal{O}$. 

\paragraph{}
We originally also claimed to answer the following question affirmatively.

\quest{
The property $\mathcal{O}$ is independent of the metric $M$.
} 

If $M$ is a Riemannian manifold which admits an orthotope partition, we can 
look at the $1$-skeleton complex $S$ of this. Obviously the property to be in $\mathcal{O}$
only refers to $S$. We can deform the metric arbitrarily
in the interior of each maximal simplex as this does not change $S$ and so the
angles between various elements in $S$. Now, the question is whether we can
translate the grid $S$ on $M$ without destroying the orthogonality condition. 
If that is the case, then $\mathcal{O}$ is independent of the metric. 
The reason is that the space $\mathcal{G}$ of Riemannian metrics
on a compact differentiable manifold $M$ is known to be a convex cone in the space
of symmetric covariant 2-tensors \cite{Ebin1968}. Therefore, invariance under
local deformations implies invariance in general. So, the answer of the above question 
really depends on whether we can shift along $S$ on $M$ by adapting the angles to 
be right angles everywhere. At this moment it is still also possible that
all manifolds are in $\mathcal{O}$ in which case the above question would also be answered
with a yes. 

\paragraph{}
To be in $\mathcal{O}$ is invariant under the operation of taking disjoint union,
products. If the property $\mathcal{O}$ is a metric invariant, it is likely also
invariant under the suspension operation. It should then also be invariant
under the operation of taking connected sums and possibly by taking fibre bundles
where both the base and the fibre are orthotope. 
A good example to test this out is $SU(3)$ which is a $S^3$ bundle over $S^5$. 
Because this bundle is not trivial, it is not clear to us however how one can chop
up the bundle in a compatible way. Since $SU(3)$ is so well known and has a bi-invariant
metric, something explicit could be written down. In a Lie group $M$, the question can 
also be seen as a crystallographic problem of building a rectangular crystal in $M$ 
partitioning $M$ into orthogonal structures. As the sphere shows, we do not need all
parts to be rectangular, there are partitions of spheres into orthogonal triangles. 

\paragraph{}
The question is also interesting in the hyperbolic case, where {\bf Escher type figures}
need to be drawn in higher dimensions.
It would be nice to have a result proving that if the universal cover is in $\mathcal{O}$ 
then $M$ is in $\mathcal{O}$. This would then apply for manifolds with negative
or non-negative curvature which by a theorem of Hadamard-Cartan have $\mathbb{R}^{2d}$ as the universal
cover. Also if $M$ is a compact manifold with a universal cover $\mathbb{R}^{2d}$
we can lift the curvature to its cover and  ask whether a ``time average"  exists 
$$  \int_M K \; dV = \lim_{R \to \infty} \int_{B_R(x)} K \; dV/V(B_R) $$
where $B_R(x)$ is a ball of radius $R$ centered around a point $x$ and $V(B_R)$ is the volume. 

\paragraph{}
Coming back to the Hopf conjectures, the negative curvature case is 
problem 10) in the collection \cite{YauSeminar1982}. As mentioned there, looking at the universal cover had been 
suggested by Singer. Also the Hopf conjecture version with $(-1)^d \chi(M) \geq 0$ for non-positively 
curved manifolds (one of the first places, where this is stated explicitly in print is 
\cite{BishopGoldberg,Chern1966}) could be studied by looking at the universal cover. 
The next paragraph gives a bit of a historical background about the Hopf conjectures. 

\paragraph{}
The {\bf Hopf sign conjecture} \cite{Hopf1932,BishopGoldberg,Chern1966,YauSeminar1982} states that a
sectional {\bf curvature sign} $e$ leads to $e^d \chi(M)>0$ for any compact Riemannian $(2d)$-manifold $M$. 
According to \cite{BergerPanorama,Berger2002}, Heinz Hopf already in the 1920ies envisioned that some kind of
Gauss-Bonnet theorem could prove this. In \cite{Weinstein71}, the term algebraic Hopf conjecture was used.
The multi-dimensional Gauss-Bonnet-Chern theorem 
worked in the case $d=2$ \cite{Chern1966}, but \cite{Geroch} showed that the Gauss-Bonnet-Chern 
integrand can become negative at some points even in the positive curvature case. We can revive the
algebraic Hopf conjecture however because there are plenty of other curvatures: take any probability 
space $(\Omega, \mathcal{A},P)$ of Morse functions and take as curvature the expected 
Poincar\'e-Hopf index density $K(x)=\E[i(x)]$.
Taking expectation of the Poincar\'e-Hopf theorem $\sum_x i_f(x)$ leads so to a generalized Gauss-Bonnet result.
Already the classical Gauss-Bonnet-Chern integrand is an index expectation curvature $\E[i_f(x)]$.

\paragraph{}
If a $2d$-manifold $M=M_1 \times \cdots \times M_d$ is a product of $2$-manifolds, 
the curvature is $K(x) = \prod_{k=1}^d K_k(x)$, the product of the Gauss curvatures $K_k$
of $M_k$. This motivates to use locally a product probability space of Morse functions so that the 
curvature becomes a product of sectional curvatures. The product case \cite{Hopf1953} 
is encouraging since this curvature $K$ is positive, even so the sectional curvatures are 
only non-negative. Despite the fact that in the product case, there are always $2$-planes with 
zero sectional curvature, the product curvature $K$ is positive. 
This is encouraging, but most manifolds are not product manifolds and the above probabilistic
trick to get a curvature like in product space at a point only works locally. 

\paragraph{}
The positivity argument involving a product probability space can be applied locally to a 
Riemannian polyhedron. We were already in \cite{DiscreteHopf,DiscreteHopf2} concerned
that the process could just move any negative curvature contribution to the boundary. And this
problem remains to be solved. We need to understand the boundary curvatures and especially 
what happens if we glue pieces of Riemannian polyhedra together. 
It can be called a {\bf Allendoerfer-Weil gluing problem} because in some sense, 
this is close to what Allendoerfer and Weil had to do when proving the Gauss-Bonnet-Chern theorem. 
One should definitely also appreciate the early spear-heading
arguments of Hopf (who could generalize Gauss-Bonnet to hypersurfaces) and Fenchel who worked
(under difficult circumstances in Europe) independently from Allendoerfer using local embeddings. 
With a Nash embedding theorem, already then, the Gauss-Bonnet-Chern theorem would have followed. 
Fenchel still matters today:
the topic of {\bf convex analysis} which was a specialty 
of Werner Fenchel \cite{BonnesenFenchel} could play a role. 

\paragraph{}
It was the Allendoerfer-Weil collaboration which first proved the generalized Gauss-Bonnet result. 
Their work was technical and got little appreciation. \cite{Wu2007} for example tells 
``such a combinatorial argument is entirely unsatisfactory from a geometric standpoint". 
Of course, the intrinsic global proof of Chern \cite{Chern44} is much more elegant and ultimate.
This is with some humor expressed by Chern himself in \cite{Chern1979}, ``the danger in cutting a manifold 
is that it might be killed". Still, the ``divide and conquer" approach clarifies why the 
product curvature works. While Chern's proof \cite{Chern44} as well as Patodi proof \cite{Cycon} 
(the proof I was exposed to as a student) are elegant and global, the combinatorial proof given here only needs 
modest tools like multi-variable calculus and does not even use differential forms. It has potential 
for more, as one can adapt the probability spaces depending on the local structure of the manifold to 
get other results. It turns out that the probability spaces in different simplices of $M$ can be modified 
arbitrarily and Gauss-Bonnet still works globally, but only if the gluing works. It is here
where the orthotope condition comes in. The {\bf Allendoerfer-Weil} gluing procedure works because
$K(x)$ has the same vertex curvature contribution than $K_{GBC}$. 

\paragraph{}
It is interesting what happens during gluing. It turns out that it is possible to glue different 
manifolds equipped with different probability spaces together without leaving traces of the 
``glue curvature" as long as some self-duality conditions are true at the boundary points and if the
probability measures are isotropic, meaning that they are invariant under orthogonal transformations in $E$. 

\paragraph{}
A critical point at the boundary of a manifold with
boundary $\delta M$ is a solution to a {\bf Lagrange problem}. The normal vector is perpendicular to the boundary 
and only if pointing inside the {\bf dual Fenchel cone} $\hat{C}$ of the solid cone $C=B_r(x) \cap M$ 
contributes an index at the boundary. It is important that the Fenchel cone is either a half space or 
that the dual Fenchel cone agrees with the Fenchel cone, because this
assures no overlap of curvature. In general, it can happen that a Morse function on one cell also contributes
curvature to a neighboring cell. This depends on the choice of probability spaces in the two cells. 

\paragraph{}
If we glue two cells, the relevant critical points on both cells $M_j$ then both contribute to critical points 
inside the intersection $M_0$. We need the {\bf isotropy condition} of the probability spaces and
that the dual Fenchel cone agrees with the Fenchel cone in order that the boundary 
curvatures from different chambers produce the Euler characteristic of the intersection allowing to leave 
away the simplices in the interior of the intersection as well as the interior boundary curvatures. 
This requires that we can chop up the manifold into orthotope blocks (which not necessarily need to 
be hyper rectangles; the $2$-sphere for example can be chopped up into 8 rectangular triangles or into two balls.) 
In general, when disregarding the orthogonality condition,
there is not always a perfect cancellation due to overlap. 
If the isotropy condition is not uniform, we also have not exact
cancellation as the curvatures at a point might not add up to $1$. 

\section{Examples}

\paragraph{}
Let us start with the example of a 4-torus $M$ with coordinates $(x,y,s,t) \in \mathbb{R}^4/(2\pi \mathbb{Z})^4$
and metric $g = dx^2+dy^2+\exp(2 u) ds^2 + \exp(-2u) dt^2$ which has the volume $4$-form $dV=dxdydsdt$. The 
function $u(t,s)$ is an arbitrary smooth function of two variables that is $(2\pi)$-periodic both in $s$ and in $t$. 
\footnote{This computation was shown to me by Cliff Taubes on 5/27/2020. 
It disproves my original claim that $\gamma_d$ is independence of the metric. }
If $K_{ij}(x,y,s,t)$ be the sectional curvature in the $ij$-plane spanned by $\{e_i,e_j\}$ with $e_j$ denoting
the standard basis, then we have a symmetric {\bf sectional curvature matrix} given by
$$ \left[ \begin{array}{cccc} 
             0             &     0          & -u_t^2-u_{tt} & -u_t^2+u_{tt} \\
             0             &     0          & -u_s^2-u_{ss} & -u_s^2+u_{ss} \\
             -u_t^2-u_{tt} &  -u_s^2-u_{ss} &   0           &  u_t^2+u_s^2  \\
             -u_t^2+u_{tt} & -u_s^2+u_{ss}  & u_t^2+u_s^2   &     0      \end{array} \right]  \; . $$
Now, summing over all $24$ permutations $\sigma$ and multiplying with $C_2 = 1/(32 \pi^2)$ gives
$$  K = \frac{1}{32 \pi^2} \sum_{\sigma} K_{\sigma(1) \sigma(2)} K_{\sigma(3) \sigma(4)} 
      = \frac{1}{2\pi^2} ( u_s^2 u_t^2 - u_{ss} u_{tt}) \; . $$
Thus, after integrating over $x,y$ givings a factor $4\pi^2$ gives
$$   \gamma_d(M) = \int_0^{2\pi} \int_0^{2\pi} 2 ( u_s^2 u_t^2 - u_{ss} u_{tt}) \; ds dt \; . $$
In the case $u(s,t) = \cos(s)+\cos(t)$, one obtains $\gamma_d(M) = \pi^2$ for $u(s,t) = \cos(t+s)$ one gets $-\pi^2/2$. 
When multiplying $u$ with a constant $\lambda$, then $\gamma_d(M)$ scales by $\lambda^2$. This example shows
that $\gamma_d(M)$ can take any real value even if $\chi(M)=0$. 
By the way, the Gauss-Bonnet-Chern integrand in this case is also explicit. We have
$K_{GBC} = (u_{t,s}^2 - u_{tt} u_{ss})/(2\pi^2)$. The Monge-Amp\`ere partial differential
equation which solves the inverse problem and gives $u$ as a function of a given curvature 
function is now even simpler than in the case of a hypersurface $z=u(r,s)$ in $\mathbb{R}^3$, 
where the Gauss curvature is $\det(d^2 u)/(1+|du|^2)^2$. 
Also as a remark, the example can be extended to $g = \exp(2 v) dx^2+\exp(-2v) dy^2+\exp(2 u) ds^2 + \exp(-2u) dt^2$
where additionally, a function $v(s,t)$ is given. The curvature expressions are still explicit but complicated
and $\gamma_d(M)$ needs to be computed numerically. 

\paragraph{}
If $M$ is the round $2d$-sphere $S^{2d}$ or radius $1$, its volume is $V=|S^{2d}|=(4 \pi)^d 2d!/(2d)!$
and its curvature is $K=(2d)!/((4\pi)^d d!)$. Multiplying $K$ with $V$
gives $2$. For $d=2$ for example, where we have a $4$-sphere
of constant curvature $K=(3/8)/\pi^2$, multiplying the curvature with the volume $|S^4|=8 \pi^2/3$ 
gives $\chi(M)=2$. One could use $\chi(M)=\gamma(M)$ 
derive the formula for $|S^{2d}|$ because all 
$(2d)!$ curvature combinations $\prod_{k=1}^d K_{\pi(2k-1),\pi(k)}$
are $1$. Now compare to $(\mathbb{S}^2)^d$, where there are $2^d d!$ 
non-zero combinations. 
We also know that $\chi(S^{2d})/\chi((\mathbb{S}^2)^d) = 2^{d-1}$ and $|(\mathbb{S}^2)^d|=(4\pi)^d$. 
Therefore, 
$$  |\mathbb{S}^{2d}|=(4\pi)^d \frac{2^d d!}{(2d!)} 
      \frac{1}{2^{d-1}} = 2d! \frac{(4\pi)^d}{(2d)!} \; . $$
The reason why the volume of $(\mathbb{S}^2)^d$ explodes exponentially in $d$, 
and the volume of $\mathbb{S}^{2d}$ decays exponentially in $d$ is that the curvature 
comparison $K_{(\mathbb{S}^2)^d}/K_{\mathbb{S}^{2d}} = 2^d d!/(2d)!$
goes much faster to zero than $\chi((\mathbb{S}^2)^d)/\chi(\mathbb{S}^{2d}) = 2^{d-1}$ 
goes to infinity. 

\paragraph{}
For $d=2$ for example, we have the parameterization 
$r(t,s,u,v)$ = $[a \cos(t)$, $\cos(s)\sin(t)$, $\sin(s)\sin(t)\cos(u)$, $\sin(s)\sin(t)\sin(u)\cos(v)$, 
$\sin(s)\sin(t)\sin(u)\sin(v)] \in \mathbb{R}^5$ which for $a=1$ parametrizes the 
round $4$-sphere $\{|x|^2=1 \} \subset \mathbb{R}^5$ and in general produces a {\bf $4$-ellipsoid 
of revolution} for which all the curvatures
and computations still can be done and where $K=C_4(8a^4)[1+a^2+(a^2-1)\cos(2t)]^{-3}$ and
$dV=\sin(s)^2 \sin(t)^3 \sin(u) [1+a^2+(a^2-1) \cos(2t)]^{1/2} dsdtdudv$. Already changing an other 
parameter requires to compute orthonormal coordinate frames. The curvature expressions for a general 
$4$-ellipsoid are complicated already and a computer algebra system balks at computing the integral. 
The curvature tensor, when written out on file contains gigabytes. 
We computed $K(x)$ numerically in a rather general $4$-ellipsoid by evaluating the 
curvature numerically at $25^4$ points. This computation takes a long time because
large trig expressions need to be evaluated. Each evaluation takes a machine a few seconds and a
full would take weeks. Using symmetry one can push it down to days. When computing things for
$\gamma$ rather than $\gamma_d$, the computations would be even more difficult.

\paragraph{}
Already for the $4$-sphere $M$, implemented as a round $4$-sphere in $\mathbb{R}^5$, 
the Gauss-Bonnet-Chern curvature $K_{GBC}$ is complicated before simplification.
For an ellipsoid we can not simplify even. For the round $4$-sphere, we can
and have $dV= \sin^2(s) \sin^3(t) \sin(u)$. After simplification, this is
$K_{GBC} dV = 96 dV $, such that 
$\int_0^{\pi} \int_0^\pi \int_0^\pi \int_0^{2\pi} K_{GB} dV/(128 \pi^2) = 2$. 
The curvature $K_d$ involving sectional curvatures is much simpler but it depends
on how we parametrize the sphere. 

\paragraph{}
For the $6$-sphere $\mathbb{S}^6$, already the usual parametrization by $6$ Euler angles $\phi_1,\dots,\phi_6$ with 
$dV=\prod_{k=1}^5 \sin(\phi_k)^k$ leads to a GBC sum which has $720^2$ terms and is hard to compute
for computer algebra system. We know by symmetry however that $K_{GBC}=2$ in the case 
of a $2d$-sphere. 

\paragraph{}
For the ellipsoid $E: x^2/a^2+y^2/b^2+z^2/c^2=1$, the Gauss curvature is 
$K(x,y,z)=a^6 b^6 c^6$ $[a^4 b^4 z^2+a^4 c^4 y^2+b^4 c^4 x^2]^{-2}$.
For the parametrization $r(\theta,\phi) =[a \sin\phi \cos\theta,b \sin\phi \sin\theta,c \cos\phi]$,
the volume form $dV$ satisfies
$dV^2 = {\rm det} (dr^T\cdot dr) = c^2 \sin^4\phi(a^2 \sin^2\theta+b^2 \cos^2\theta)+a^2 b^2 \sin^2\phi \cos^2\phi$.
so that
$K dV = a^2 b^2 c^2$ $\sqrt{c^2 \sin^4\phi [a^2 \sin^2\theta+b^2 \cos^2\theta]+a^2 b^2 \sin^2\phi \cos^2\phi}$
             $[c^2 \sin^2\phi (a^2 \sin^2\theta+b^2 \cos^2\theta)$+$a^2 b^2 \cos^2\phi]^{-2}$. 
This doable integral evaluates to $\int_0^{2\pi} \int_0^{\pi} K(\theta,\phi)$ $dV(\theta,\phi)$ \; $d\phi d\theta = 2$.
The $4$-manifold $M=E \times E$ produces a Riemannian metric on $S^2 \times S^2$. At a particular
point $r(\theta_1,\phi_1,\theta_2,\phi_2)$ this gives an orthonormal frame $t=(t_1,t_2,t_3,t_4)$, 
where $t_1=r_{\theta_1}/|r(\theta_1)|$ and $r_3=r_{\theta_2}/|r(\theta_2)|$ and 
$t_{2k}=(r_{\phi_k} - (r_{\phi_k} \cdot t_{2k-1}) t_{2k-1})/|r_{\phi_k} - (r_{\phi_k} \cdot t_{2k-1}) t_{2k-1}|$.
The sectional curvatures $K_{ij}$ are all zero except for $K_{12}=K_{21}=K(\theta_1,\phi_1)$,
$K_{34}=K_{43}=K(\theta_2,\phi_2)$. The volume measure of the $4$-manifold $M$ is $dV(\theta_1,\phi_1) dV(\theta_2,\phi_2)$. 
The integral $\int_K \; dV$ can be split and gives $(\int_{E} K \; dV)^2 = 4$. 

\paragraph{}
The real projective plane $\mathbb{P}^2$ 
can be parametrized 
(see \cite{BergerGostiaux} page 89) in $E=\mathbb{R}^6$ as
$r(t,s)$ = $[\sin^2(s) \cos^2(t)$,  $\sin^2(s) \sin^2(t)$, $\cos^2(s)$, 
     $\sqrt{2} \sin(s)\cos(s)\sin(t)$,  $\sqrt{2} \sin(s)\cos(s)\cos(t)$, \linebreak
     $\sqrt{2}\sin^2(s)\sin(t)\cos(t)]$. 
The metric is $\left[ \begin{array}{cc} 2\sin^2(s) & 0 \\ 0 & 2 \end{array} \right]$
and the curvature is constant $1/2$. With that parametrization, the volume 
$4\pi$ is the same than the volume of the sphere $\mathbb{S}^2$. The Euler characteristic
is half of the Euler characteristic of $S^2$ and equal to $1$ as it should be by the
Riemann-Hurwitz covering formula. Of course, as $M$ is a 2-manifold, 
we have $\gamma_d(M)=\gamma(M)=\chi(M)$. 

\paragraph{}
The volume of the product manifold $M=(S^2)^d$ is $(4 \pi)^d$. Its 
Euler characteristic is $\chi(M)=2^d$. 
The volume ratio between $(S^2)^d$ and $S^{2d}$ is $2^d d!/(2d)! = (2d-1)!!$, the double factorial. 
From the $(2d)!$ sectional curvatures, there are $2^d d!$ which are non-zero. So, using the product
frame, $K_d=(2d)! 2^d d!/(2d)! (1/(d! (4\pi)^d)$. Multiply this with the volume $V(M)=(4 \pi)^d$ gives 
the Euler characteristic $\chi(M)=2^d$. 

\paragraph{}
$M=M_1 \times M_2$ of arbitrary odd-dimensional manifolds $M_i$. While also $K_{GBC}$ is zero in this case,
it is easier to see it for the curvature $K_d$ using a product compatible frame,
as the later involves sectional curvatures only. However, it is important here that the 
coordinate frame is chosen compatible with the decomposition. 
Examples are the $6$-manifold $M=S^3 \times S^3$ where each component $K_1 K_2 K_3$ has one part which splits
and so gives curvature zero.  Related examples are $P^3 \times S^3$ or $P^3 \times P^3$ or $S^1 \times N$, where
$N$ is an odd dimensional manifold. 

\paragraph{}
Given a $2d$-manifold, $M=M_1 \times M_2$, where both factors $M_i$ are even dimensional.
If a point $(x,y)$ is given, this gives a natural product coordinate system. 
While the Riemann curvature tensors (and so sectional curvature) of the product however add, we have
$K(x,y) = K_1(x) K_2(y)$, where $K_i$ are the curvatures in $M_i$.
The curvature is $K=K_1 K_2$, where $K_1$ is the curvature in $M_1$ and $K_2$ is the curvature in $M_2$. 
The Euler characteristic $\chi(M)$ is the product $\chi(M) = \chi(M_1) \chi(M_2)$. When seeing this 
probabilistically, the curvature is the expectation of indices of Morse functions and this relation is $E[i_{f_1 + f_2}] 
= E[i_{f_1}] E[i_{f_2}]$ which is a manifestation of $E[X Y] = E[X] E[Y]$ for independent random variables. 

\paragraph{}
The {\bf complex projective plane} $M=\mathbb{CP}^2$ is a 4-manifold homeomorphic to the 
quotient $S^{5}/S^1$. It is naturally equipped with the Fubini-Study metric $g$. 
The sectional curvature for holomorphic planes is $4$, the
other perpendicular planes have sectional curvature $1$. 
This follows from the formula $K_{X,Y} = 1+3 g(JX,Y)$ given in \cite{DoCarmo1992},.
Jason de Vito sent us a computation for the discrete
version of $M$ showing that that functional can be coordinate dependent. (We originally thought
it would be independent of the coordinates which is not true). 
The volume of $\mathbb{CP}^d$ is $V(\mathbb{CP}^d) =\pi^d/d!$ \cite{BergerPanorama}.
The sum over all curvatures is for $4$-manifolds is
$$ \sum_\pi K_{\pi(1) \pi(2)} K_{\pi(3) \pi(4)} = 8( K_{12} K_{34} + K_{13} K_{24} + K_{14} K_{23} )  \; . $$ 
For the ortho-normal basis $(1,0)$, $(i,0)$, $(0,1)$, $(0,i)$, the do Carmo formula gives
$K_{12} = K_{34} = 4$ as these are the holomorphic planes, 
while all other sectional curvatures are $K_{ij} = 1$. This gives the sum $8 (16+2)=144$. 
On the other hand, for the basis $(1,1)/\sqrt{2}$, $(i,0)$, $(1,-1)/\sqrt{2}$, $(0,i)$, one has
$K_{13} = K_{23} = 1$, with all other $K_{ij} = (1+3/2) = 5/2$. This gives  $8(1+25/4+25/4) = 108$. 
Projective spaces are important in physics because the $2d$-manifold $\mathbb{CP}^d$ is the space 
of pure states of a spin $d/2$-particle, the case $d=1$ being known as the {\bf Bloch sphere}.

\paragraph{}
Here is an illustration of Gauss-Bonnet for manifolds with boundary. 
If $M$ is $2$-sphere with metric $g$, cut it by a plane $\Sigma=\{ z=\cos(\phi_0) \} $. 
This produces two pieces $M^{\pm}$ which are topological discs with boundary and non-trivial boundary curvatures.
The inner total curvature of the top part is $1-z$ and of the lower part $1+z$. The boundary curvature
of the top is $z$ and for the bottom $-z$. Gauss-Bonnet gives in each part the Euler characteristic $1$.
The boundary manifold $M^0$ is a closed circle. It has an
intrinsic curvatures $k^{\pm} = \pm z$ from each side. The total curvature of the intersection zero. 
For getting the curvature $k^{\pm}$, we look at functions for which the gradient points in the half space $E^{\pm}$. 
Critical points from one side come by Lagrange by critical points under a constraint $\Sigma$ meaning
that the gradient is perpendicular to $\Sigma$. When looking at critical points 
on the intersection $M \cap \Sigma$, we get the union of both critical points. The index sum
in that case is $\chi(M^0)$. 

\paragraph{}
If $M=G$ is a Lie group equipped with a bi-invariant metric, and $X,Y,U,V$ are left-invariant vector fields
then $4R(X,Y,U,V)=g([X,Y],[U,V])$ so that $G$ has non-negative sectional curvature $K(X,Y)=|[X,Y]|^2$ 
(i.e. \cite{Gallot} section 3.17 or \cite{Mil63} Theorem 21.3). 
The Euler characteristic of a Lie-Group is always zero due to the existence of a non-zero vector field. 
$SO(4)$ for example has dimension $6$ and is doubly covered by ${\rm Spin}(4)=\mathbb{S}^3 \times \mathbb{S}^3$ 
having so only a semi-simple, not a simple Lie algebra. 
The metric $g(X,Y)=-{\rm tr}({\rm ad}_X {\rm ad}_Y)/2$, where ${\rm tr}({\rm ad}_X {\rm ad}_Y)$ 
is the Killing form is the standard metric in $\mathbb{R}^6$. 
It is non-degenerate as it is for any semi-simple Lie group by Cartan's Criterion and positive definite as $G$ is compact. 
In a basis which is compatible with the factorization $so(4) = so(3) \times so(3)$, the metric is diagonal. 
An orthonormal basis for $so(4)$ is given by the $6$ matrices:
\begin{tiny}
$$
\{ \left[\begin{array}{cccc}0&0&0&-\frac{1}{2}\\0&0&-\frac{1}{2}&0\\0&\frac{1}{2}&0&0\\\frac{1}{2}&0&0&0\\\end{array}\right],
\left[\begin{array}{cccc}0&0&\frac{1}{2}&0\\0&0&0&-\frac{1}{2}\\-\frac{1}{2}&0&0&0\\0&\frac{1}{2}&0&0\\\end{array}\right],
\left[\begin{array}{cccc}0&-\frac{1}{2}&0&0\\\frac{1}{2}&0&0&0\\0&0&0&-\frac{1}{2}\\0&0&\frac{1}{2}&0\\\end{array}\right] $$
$$ \left[\begin{array}{cccc}0&0&0&\frac{1}{2}\\0&0&-\frac{1}{2}&0\\0&\frac{1}{2}&0&0\\-\frac{1}{2}&0&0&0\\\end{array}\right],
\left[\begin{array}{cccc}0&0&\frac{1}{2}&0\\0&0&0&\frac{1}{2}\\-\frac{1}{2}&0&0&0\\0&-\frac{1}{2}&0&0\\\end{array}\right],
\left[\begin{array}{cccc}0&-\frac{1}{2}&0&0\\\frac{1}{2}&0&0&0\\0&0&0&\frac{1}{2}\\0&0&-\frac{1}{2}&0\\\end{array}\right] \} \; . $$
\end{tiny}

\paragraph{}
The manifold $M=SU(3)$ is a simple Lie group and produces an $8$-dimensional Riemannian manifold.
A basis in the Lie algebra is given by the 
Gell-Mann matrices. In order to get the metric and so to compute the curvature, we 
need to compute the matrices of an adjoint representation first. The formula for the sectional curvatures
for a left-invariant metric are given in \cite{Milnor1976} 
in terms of the structure constants $\alpha_{ijk}$ of $M$. In the concrete basis of Gell-Mann matrices 
\begin{tiny}
$$ \{ 
\left[ \begin{array}{ccc} 0 & 1 & 0 \\ 1 & 0 & 0 \\ 0 & 0 & 0 \\ \end{array} \right],
\left[ \begin{array}{ccc} 0 & -i & 0 \\ i & 0 & 0 \\ 0 & 0 & 0 \\ \end{array} \right],
\left[ \begin{array}{ccc} 1 & 0 & 0 \\ 0 & -1 & 0 \\ 0 & 0 & 0 \\ \end{array} \right],
\left[ \begin{array}{ccc} 0 & 0 & 1 \\ 0 & 0 & 0 \\ 1 & 0 & 0 \\ \end{array} \right],  $$
$$ 
\left[ \begin{array}{ccc} 0 & 0 & -i \\ 0 & 0 & 0 \\ i & 0 & 0 \\ \end{array} \right],
\left[ \begin{array}{ccc} 0 & 0 & 0 \\ 0 & 0 & 1 \\ 0 & 1 & 0 \\ \end{array} \right],
\left[ \begin{array}{ccc} 0 & 0 & 0 \\ 0 & 0 & -i \\ 0 & i & 0 \\ \end{array} \right],
\left[ \begin{array}{ccc} \frac{1}{\sqrt{3}}&0&0\\0&\frac{1}{\sqrt{3}}&0\\0&0&-\frac{2}{\sqrt{3}}\\\end{array} \right] \}$$
\end{tiny}
the structure constants $\alpha_{ijk}$ are and either $\pm 1, \pm 1/2, 0, \pm \sqrt{3}/2$. 
With Milnor's formula
$$  K_{ij} = \sum_k \frac{1}{2} \alpha_{ijk}(-\alpha_{ijk}+\alpha_{jki}+\alpha_{k12}) 
                -\frac{1}{4} (\alpha_{ijk}-\alpha_{jki}+\alpha_{kij})(\alpha_{12k}+\alpha_{2k1}-\alpha_{k12})
                -\alpha_{k11} \alpha_{k22}  \; . $$
The sectional curvatures are the entries $K_{ij}$ of
$$ K = \left[ \begin{array}{cccccccc}
    0 & \frac{1}{4} & \frac{1}{4} & \frac{1}{16} & \frac{1}{16} & \frac{1}{16} & \frac{1}{16} & 0 \\
    \frac{1}{4} & 0 & \frac{1}{4} & \frac{1}{16} & \frac{1}{16} & \frac{1}{16} & \frac{1}{16} & 0 \\
    \frac{1}{4} & \frac{1}{4} & 0 & \frac{1}{16} & \frac{1}{16} & \frac{1}{16} & \frac{1}{16} & 0 \\
    \frac{1}{16} & \frac{1}{16} & \frac{1}{16} & 0 & \frac{1}{4} & \frac{1}{16} & \frac{1}{16} & \frac{3}{16} \\
    \frac{1}{16} & \frac{1}{16} & \frac{1}{16} & \frac{1}{4} & 0 & \frac{1}{16} & \frac{1}{16} & \frac{3}{16} \\
    \frac{1}{16} & \frac{1}{16} & \frac{1}{16} & \frac{1}{16} & \frac{1}{16} & 0 & \frac{1}{4} & \frac{3}{16} \\
    \frac{1}{16} & \frac{1}{16} & \frac{1}{16} & \frac{1}{16} & \frac{1}{16} & \frac{1}{4} & 0 & \frac{3}{16} \\
    0 & 0 & 0 & \frac{3}{16} & \frac{3}{16} & \frac{3}{16} & \frac{3}{16} & 0 \\ \end{array} \right] \; . $$
Already $K_{12} K_{34} K_{56} K_{78} = 3/16384$ is non-zero. But the Euler characteristic is $0$. 
The volume of $SU(3)$ is given as 
$V(SU(3)) = Pi^5$. We measure the sum over all curvature quadruples to be $351/64$.
We have also $c_4 = \frac{1}{6144 \pi ^4}$. Therefore $\gamma_d(M)=117 \pi/2^17$.
Also here, the curvatures and $\gamma_d$ depend on the frame. 
If we make a rotation of the basis of $su(3)$ in the adjoint representation, 
then we measure in general different values. We have not yet attempted to 
compute the true $\gamma(M)$ as it requires to integrate over part of the rotation group.

\paragraph{}
The {\bf Klembeck example} \cite{Klembeck} is a simplification of the Geroch example \cite{Geroch}. 
Look at a patch of a 6-manifold with metric 
$\left[
                  \begin{array}{cccccc}
                   1-3 z^2 & -2 u z & 0 & 0 & 0 & 2 v y \\
                   -2 u z & 1-3 u^2 & 2 u x & 0 & 0 & 0 \\
                   0 & 2 u x & 1-3 v^2 & -2 v w & 0 & 0 \\
                   0 & 0 & -2 v w & 1-3 w^2 & 2 w z & 0 \\
                   0 & 0 & 0 & 2 w z & 1-3 x^2 & -2 x y \\
                   2 v y & 0 & 0 & 0 & -2 x y & 1-3 y^2 \\
                  \end{array}
                  \right]$. A direct computation gives
at $(x,y,z,u,v,w)=(0,0,0,0,0,0)$ the sectional curvatures $K=\left[ \begin{array}{cccccc}
                   0 & 0 & 3 & 0 & 3 & 0 \\
                   0 & 0 & 0 & 3 & 0 & 3 \\
                   3 & 0 & 0 & 0 & 3 & 0 \\
                   0 & 3 & 0 & 0 & 0 & 3 \\
                   3 & 0 & 3 & 0 & 0 & 0 \\
                   0 & 3 & 0 & 3 & 0 & 0 \\
                  \end{array} \right]$. Obviously, the curvature $K$ is non-negative here.
The Gauss-Bonnet-Chern integrand is $K_{GBC}=-9216/(6!)^2<0$. 
It is now possible to change the metric a bit to make it positive at $0$ while still
keeping $K_{GBC}$ negative. 

\section{Poincar\'e-Hopf}

\paragraph{}
Let $M$ be a $2d$-dimensional {\bf Riemannian polytop}. That means that $M$ is a
$2d$-dimensional compact Riemannian manifold with piecewise smooth boundary 
$\delta M$. We will later use a triangulation to decompose a general smooth 
compact Riemannian manifold into small contractible Riemannian polytopes having
the shape of simplices, but for much of what we doing here, no contractibility of 
$M$ is necessary; the manifold $M$ is just a Riemannian manifold with 
piecewise smooth boundary. 

\paragraph{} 
By the Nash's embedding theorem \cite{EssentialNash}, we can assume that $M$ is 
isometrically embedded in a larger dimensional Euclidean space $E$ and that each part of 
the boundary $\delta M$ is given as an intersection of 
regular level surfaces $\{ g_k=c \}$. 
Smooth Riemannian manifolds can always be triangulated 
(this only becomes difficult in the topological category) so that $M$ can be seen as a 
{\bf geometric realization} of a finite abstract simplicial 
complex of dimension $2d$ which is {\bf pure} in the sense that all maximal simplices have dimension $2d$. 
Without much loss of generality, we can restrict our discussion at first to the case when $M$ is a simplex. 
We will then later glue together such local {\bf cells}.

\paragraph{}
So, in order to fix the ideas, we assume that the Riemannian polytop 
$M$ is a smooth image $r(\Delta)$ of a {\bf $(2d)$-simplex}
$\Delta = \{ x \in \mathbb{R}^{2d+1}, x_i \geq 0$,
for all $i =0, \dots ,2d$ and $\{ \sum_{i=0}^{2d} x_i=1 \}$, where 
$r: \mathbb{R}^{2d+1} \to E$ is smooth. 
This produces a parametrization of $M$ by $\Delta$.  

\paragraph{}
Alternatively, it can be useful to look at the case, where
$M$ is a smooth image $r(Q)$ of a standard {\bf 2d-cuboid}
$Q = \{ x \in \mathbb{R}^{2d},  0 \leq x_i \leq 1 \}$ 
in $\mathbb{R}^{2d}$. A triangulation can be
obtained from such a cuboid 
by triangulating $Q$ into simplices, for example by a {\bf Freudenthal 
triangulation} \cite{DangTriangulations}. We especially have so
explicit triangulations of $\mathbb{R}^{2d}$ which are by Cartan-Hadamard
the universal covers of negatively curved manifolds or non-positively
curved manifolds. 

\paragraph{}
In the case of a triangulation, the simplicial complex structure 
of $\Delta$ goes over to $M$ so that to every point $x$, 
one can associate a {\bf dimension of the point}.
It is the largest dimension $k$ such that $x$ is in the interior of the image 
of the {\bf $k$-dimensional sub simplex} = {\bf k-face}
$\Delta_k \subset \Delta$ under the smooth map $r$. 
Having a well defined notion of {\bf dimension for points} will be useful
when defining what we mean with a {\bf Morse function} on such a manifold with
boundary. 

\paragraph{}
A function $f: M \to \mathbb{R}$ is defined to be {\bf smooth} if its 
restriction to the interior of each $k$-dimensional piece 
$r(\Delta_k)$ is smooth. If $M$ is isometrically embedded in an ambient Euclidean space, 
a smooth function $f: E \to \mathbb{R}$ induces a smooth function 
$f: M \to \mathbb{R}$ by restricting the domain from $E$ to $M$.  
A point in $M$ is a {\bf critical point} of $f$, if either it is 
in the interior of $M$ and $\nabla f(x)=0$ or if it is a Lagrange 
critical point at the boundary. 

\paragraph{}
If $x$ is in the interior of a $k$-dimensional part of the boundary 
that is given by the intersection of $2d-k$ functions 
$g_1, \dots, g_{2d-k}$ which each have no critical points,
then a {\bf Lagrange critical point} means that $\nabla f$ is 
a linear combination of gradients of $g_1, \dots, g_{2d-k}$.  These
are the Lagrange equations. 

\paragraph{}
The situation of critical points for general smooth functions of a
Riemannian polytop $M$ can be quite complicated in general. As investigated 
by mathematicians like Whitney or Thom, this leads to subjects like
{\bf singularity theory} or {\bf catastrophe theory}. The Morse set-up 
tames such difficulties.  Assuming $f$ to be Morse assures for example
that critical points are isolated, non-degenerate and so stable under
perturbations. Morse theory is a fantastic set-up because it achieves
accessibility similarly than complex analytic frame-works without the 
severe constraints, which the complex analytic category produces.

\paragraph{}
Much of the complexity is reduced by assuming a Morse condition.
This has to be defined carefully if one deals with manifolds with 
boundary. So, to express this more precisely, 
we assume that if $x$ is a critical point in $M_k$, 
then it is a {\bf Morse critical point} in the interior a $k$-dimensional 
part $M_k$ of $M$, then $f$ restricted to the interior of 
$M_k$ is Morse in the classical sense. Let us rephrase this: 

\paragraph{}
A function $f: M \to \mathbb{R}$ is called a {\bf Morse function} on 
a Riemannian polyhedron $M$ if for all critical points 
in the interior $M_{2d}$ of $M$, the standard Morse condition is 
satisfied, meaning that the Hessian $d^2 f(x)$ has full rank at 
such a critical point. Additionally, we want to have Morse conditions
in the smaller dimensional parts. The definition continues 
in the next paragraph. 

\paragraph{}
If the critical point $x$ is in the interior of a $k$-dimensional 
part $M_k$ of the boundary, we assume that all Lagrange 
multipliers $\lambda_j(x), j=1,\dots,2d-k$ are non-zero and $f$ restricted 
to the interior of $M_k$ is Morse.  A critical point $x$ then satisfies 
$$  \nabla f(x) = \sum_{j=1}^{2d-k} \lambda_j(x) \nabla g_j(x) \; , $$ 
where $M_k$ is contained in  $\{ g_1(x)= \dots = g_{2d-k}(x) = 0 \}$. 
The Morse condition for a critical point $x$ in $M_k$ assumes 
that $x$ not already a critical point in any of the parts $M_l$ with $l>k$. 

\paragraph{}
The {\bf Poincar\'e-Hopf index} of a Morse function $f$ at a critical point 
$x$ in a Riemannian polyhedron $M$ is 
defined as $i_f(x)=1-\chi(S_f(x))$, where $S_f(x)$ is the part of a 
sufficiently small geodesic sphere $S_r(x)$ around $x$, where $f$ is smaller or equal 
than $f(x)$. Formally, 
$$   i_f(x) = \chi(S_f(x)) 
    = \chi( \{ y \in S_{r,M}(x) \; | \; f(y) \leq  f(x) \} ) \; , $$ 
where $S_{r,M}(x) = \{ y \in M \; | \; d(x,y)=r \; \}$ is a geodesic sphere
near $x$, the set of points in $M$ of distance $r$ to $x$, measured in terms of
the Riemannian metric given on $M$.

\paragraph{}
In the above definition, we implicitly assume that $r>0$ is chosen sufficiently 
small at a point and usually not mention the radius $r$.
If $M$ is a compact Riemannian polyhedron and $x$ is a point in $M$, 
then there is a threshold $r_0=r_0(x)$ such that all 
$S_{r,M}(x)$ are homeomorphic for $0<r < r_0(x)$.
We can not chose a lower bound $r_0$ uniformly in $M$ as a
point can be close to the boundary, where $r_0$ has to be smaller than 
the distance to the boundary. (We can however get a fixed $r_0$ for every 
Morse function $f$.)  More generally, if we have
a $k$-dimensional point $x \in M_k$, then $r_0$ depends on 
the geodesic distance to the next $(k-1)$-dimensional point. 

\paragraph{}
In the interior of $M$, every critical point $x$ of a Morse function $f$
either has index $+1$ or $-1$. At the boundary, it is possible that 
the index of a critical point is zero. 
This already happens in very simple situations like if $M = \{ x^2+y^2 \leq 1 \}$ 
is a solid disc in $E=\mathbb{R}^2$ and $f$ is a linear function like $f(x)=y$. 
Only minima of $f$ at the boundary lead to critical points and no critical points 
exist in the interior. In this particular example, there are two critical points 
$(0,-1)$ and $(0,1)$ and $i_f(0,-1)=1-\chi(\emptyset)=1$ while $i_f(0,1)=1-1=0$. 
The total sum of all indices is $1=\chi(M)$ if $M$ is a simplex as it should
be for Poincar\'e-Hopf. 

\paragraph{}
As a side remark, in the case of a convex set $M$ in $\mathbb{R}^2$ bound by a smooth 
simple Jordan curve, the curvature obtained by averaging over 
all linear functions $f_a(x) = a \cdot x$ is now technically a 
{\bf distribution} (a generalized function) and
supported on the boundary of $M$ and $K(x) = {\rm E}[ i_f(x)]$ 
is the usual normalized {\bf signed curvature} $K(x)/(2\pi)$ 
which is for $x=r(t)$ with an arc-length
parametrization given by $|r''(t)|$. The Gauss-Bonnet theorem $\int_M K \; dV = 1$ is
then called the {\bf Hopf Umlaufsatz}. By the way, the boundary $\delta M$ of the convex
region, then it is a one-dimensional manifold for which the total curvature 
is zero, the reason being that $i_f(x)=-i_{-f}(x)$ then. 
This is compatible with $\chi(\delta M)=0$. 

\paragraph{}
Here is a lemma which explains why it is useful already to assume the probability spaces of
Morse functions to be invariant under the involution $f \to -f$. 
For gluing estimates on corners were several manifolds come together we need even 
an {\bf isotropy condition}, the rotational invariance of the probability spaces
under maps $f \to Af$ with an orthogonal transformation $A$ on $E$.  

\begin{lemma}
Assume $x \in M$ is a critical point of a Morse function $f$ and $x$ has dimension $k$. 
Either $f$ or $-f$ has an index which agrees with the index of $f$ restricted to $M_k$. 
\end{lemma}
\begin{proof}
Either the gradient $\nabla f(x)$ points outside $M$ (in which case the index is zero) 
or then inside in which case the index $i(x)$ is $1-\chi(S_f^-(x))$ (in which case it can be zero or not). 
Now $S_r(x) \cap \{ f \leq 0 \} \cap M$ 
and $S_r(x) \cap \{ f \leq 0 \{ \cap \delta M$ are homotopic and so have the same 
Euler characteristic. 
\end{proof}

\paragraph{}
The simplest example which shows why the symmetry is useful is to take the 
two $1$-dimensional simplices $M_k$. Now, given any Morse functions $f_k$ on $M_k$
the index $i_f(x)$ is $0$ at one end (where $f$ is maximal) and $1$ at the other end 
(where $f$ is minimal). Assume now that  $\Omega_k = \{ f_k\}$ have just one element so
that the index is the curvature. Now glue the two segments together so that the two ends
with curvature $0$ are glued. After gluing, there is a curvature $-1$ in the interior
and curvature $1$ at the boundary. The total curvature is $1+1-1=1$ as it should be
but we have some glueing curvature in the intersection. Now, in the symmetric situation,
we have curvature $1/2$ at each end. After gluing, we have again a larger $1$-dimensional
$M$ with boundary which has curvature $1/2$ at the boundary and $0$ curvature at the
point, where the gluing has taken place. The gluing measure now has disappeared.

\paragraph{}
{\bf Remark.} The result could be generalized to more general {\bf Riemannian polyhedra}, 
manifolds which are piecewise smooth and have piecewise smooth Riemannian manifolds as boundaries. 
For us here, we only need the result, where $M$ is homeomorphic to a topological ball
and a geometric realization of a $2d$-simplex $\Delta$ inside a Euclidean space $E$. 

\paragraph{} 
{\bf Remark.} The usual assumption for {\bf Poincar\'e-Hopf for manifolds with boundary}
is that the vector field $F=\nabla f$ points outwards everywhere on the boundary.
We do not make this assumption here. It is the assumption that $f$ is Morse that makes it easier. 
This appears to be a fresh version of Poincar\'e-Hopf. Usually, the Poincar\'e-Hopf theorems
for manifolds $M$ with boundary assume that the vector field is everywhere 
perpendicular to the boundary $\delta M$ and never zero on the boundary. 
Other versions of Poincar\'e-Hopf statements for manifolds with boundary are \cite{Pugh1968}. 

\paragraph{}
Our definition of index as an Euler characteristic of a stable sphere has 
the advantage that {\bf it is the same than in the discrete} and works for 
any network (a finite simple graph with Whitney complex) and any finite abstract simplicial complex $G$.
Such a simplicial complex $G$ does not even have to be a triangulation of a manifold, it can 
be an arbitrary finite set of non-empty sets closed under the operation of taking finite 
non-empty subsets. 

\begin{thm}[Poincar\'e-Hopf]
If $f$ is a Morse function on a compact Riemannian polyhedron, 
then the sum of the indices of $f$ is the Euler characteristic of $M$. 
\end{thm}
\begin{proof}
A simple proof is obtained by making a triangulation of 
$M$ and boundary which is {\bf adapted to the function $f$}.
This means that the triangulation should also triangulate 
each sphere or half-sphere $S_r(x)$, whenever $x$ is a 
critical point of $f$. We can then use the result for finite 
abstract simplicial complexes. That argument is so simple, that
it can be repeat it in detail in the next paragraph.
Since the Euler characteristic of the graph is the same than 
the Euler characteristic of the manifold for which the graph
is the 1-skeleton of a triangulation of, and the indices are the 
same, we are done. 
\end{proof}

\paragraph{}
In order not to keep this paper self-contained (we could refer to 
\cite{KnillEnergy2020} although),
let us state the result for a {\bf finite abstract simplicial complex} $G$, 
which is a finite set of non-empty subsets of a finite set $V$ 
closed under the operation of taking non-empty subsets. 
The Euler characteristic of $G$ is 
$\sum_{x \in G} \omega(x)$, where 
$\omega(x)=(-1)^{{\rm dim}(x)} = (-1)^{|x|-1}$, where $|x|$ is the 
cardinality of $x$. For any map $F:G \to V$ satisfying $F(x) \in x$, 
the index $i_F(x) =\chi(F^{-1}v)$ satisfies 
$\sum_{x} i_F(x) = \chi(G)$ because the energy $\omega(x)$ 
can be transported along $F$ to $V$. 
Now, if $G$ is the set of complete subgraphs of a graph $(V,E)$ and 
$f: V \to \mathbb{R}$ is a function which is locally injective 
in the sense that $f(x) \neq f(y)$ if $x$ and $y$ are connected, 
then it defines $F(x)$ as the vertex in the simplex $x$, where $f$ is minimal. 
The corresponding index is the Poincar\'e-Hopf index 
$i_f(x) = 1-\chi(S_f^-(x)))$, where $S_f^-(x)$ is the graph generated by the
subset of all $y$ directly attached to $x$ and where $f(y)<f(x)$. 
The formula $\chi(G) = \sum_{v \in V} i_f(v)$ is the
Poincar\'e-Hopf relation. Now if $M$ is a Riemannian polyhedron 
and $f: M \to \mathbb{R}$ is a Morse function we have only finitely 
many critical points. Chose now a small radius $r>0$ and triangulate $M$ 
such that every sphere $S_r(x)$ is triangulated in such a way that $f$ 
is locally injective, if $x$ is a critical point. 
This now assures that the simplicial complex
$G$ of the triangulation has the Euler characteristic of $M$ 
and that the index $i_{f,M}(x)$ of every critical 
point is the index $i_{f,G}(x)$. 
Now, the Poincar\'e-Hopf theorem for the Riemannian polyhedron $M$ 
follows from the Poincar\'e-Hopf theorem for simplicial complexes $G$. 

\paragraph{}
{\bf Example:} If $M = r([a,b]) \subset E$ is a smooth immersed curve with boundary $\{ r(a), r(b) \}$ and $\Omega$ is 
the set of linear functions on $E$, then for almost all $f \in \Omega$, the induced function 
is Morse meaning that at all places where $t \to f(\vec{r(t)})$ has zero derivative, the second
derivative is non-zero. Additionally, we want that there are no critical points at the boundary. 
Now, the total index in the interior of $M$ is zero, the reason being that $f$ and $-f$ have different
indices. The index expectation at the end points is $1/2$ as for half of the functions
an end point is a minimum having index $1$. 

\paragraph{}
{\bf Example:} Let $M=r( [0,2\pi] \times [0,\phi] )$ with 
$r(\theta,\phi) = [ \cos(\theta) \sin(\phi), \sin(\theta) \sin(\phi),\cos(\phi)]$. 
The curvature induced on the interior of $M$ is $1$. The total curvature of the interior is 
the area $A/(2\pi) = (2\pi (1-\cos(\phi)))/(2\pi) = 1-\cos(\phi)$. 
The total curvature of the boundary is $\cos(\phi)$. For $\phi=0$, then all the curvature is
on the boundary. 

\paragraph{}
{\bf Example:} If $M$ is the $2d$-ellipsoid realized as 
$\sum_{i=1}^{2d+1} x_i^2/a_i^2$ in $\mathbb{R}^{2d+1}$ and $f(x) = x_{2d+1}$ 
then there are two critical points, the maximum $A=(0,0,\dots,a_{2d+1})$ and the
minimum $B=(0,0,\dots,-a_{2d+1})$. The stable sphere $S_r(A)$ is a $(2d-1)$ dimensional 
ellipsoid with Euler characteristic $0$ so that $i_f(A) = 1- \chi(S_r(A))=1$. 
The stable sphere $S_r(B)$ is empty with Euler characteristic $0$ so that also $i_f(B)=1$. 
Poincar\'e-Hopf shows that $\chi(M)=1+1=2$. This Morse function $f$ is given as $f(x) = x \cdot u$
with $u=(0,0,\dots, 1)$. It generalizes to any $f(x)=x \cdot u$, where $\sum_{i=1}^{2d+1} u_i^2=1$. 
In each case, the strict convexity of $M$ assures that there are only positive index critical points. 

\paragraph{}
{\bf Example:} The $2d$-torus $M$ can be equipped with a metric by embedded in $\mathbb{R}^{3d}$ 
using a parametrization $r(s_1,t_1,\dots ,s_d,t_d)=$ 
$[(2+\cos(s_1)) \cos(t_1),(2+\cos(s_1)) \sin(t_1)$, \dots,
                              $(2+\cos(s_d)) \cos(t_d),(2+\cos(s_d)) \sin(t_d)$ $]$. The
manifold $M$ is obviously the product of $d$ two-dimensional tori $M_k$. Unlike in the previous example,
not all functions $f(x) = x \cdot u$ are Morse. The function $f(x)=x_{3d}$ for example is not,
because the Hessian at a critical point has a $(2d-d)$-dimensional kernel. But the function 
$f(x)=\sum_{k=1}^d x_{3k}$ is Morse. There are now $4^d$ critical points of the form 
$v=(v_1,\dots, v_d)$ with $v_k \in \mathbb{R}^3$ being critical points of $f_k(x)=x_{3k}$. 
The index is $i_f(v) = \prod_{k=1}^d i_{f_k}(v_k)$. If $a_k,b_k,c_k,d_k$ are the critical points
of $f_k$, then $\chi(M_k) = i_{f_k}(a_k) + i_{f_k}(b_k) + i_{f_k}(c_k) + i_{f_k}(d_k) = 0$. 

\begin{figure}[!htpb]
\scalebox{0.4}{\includegraphics{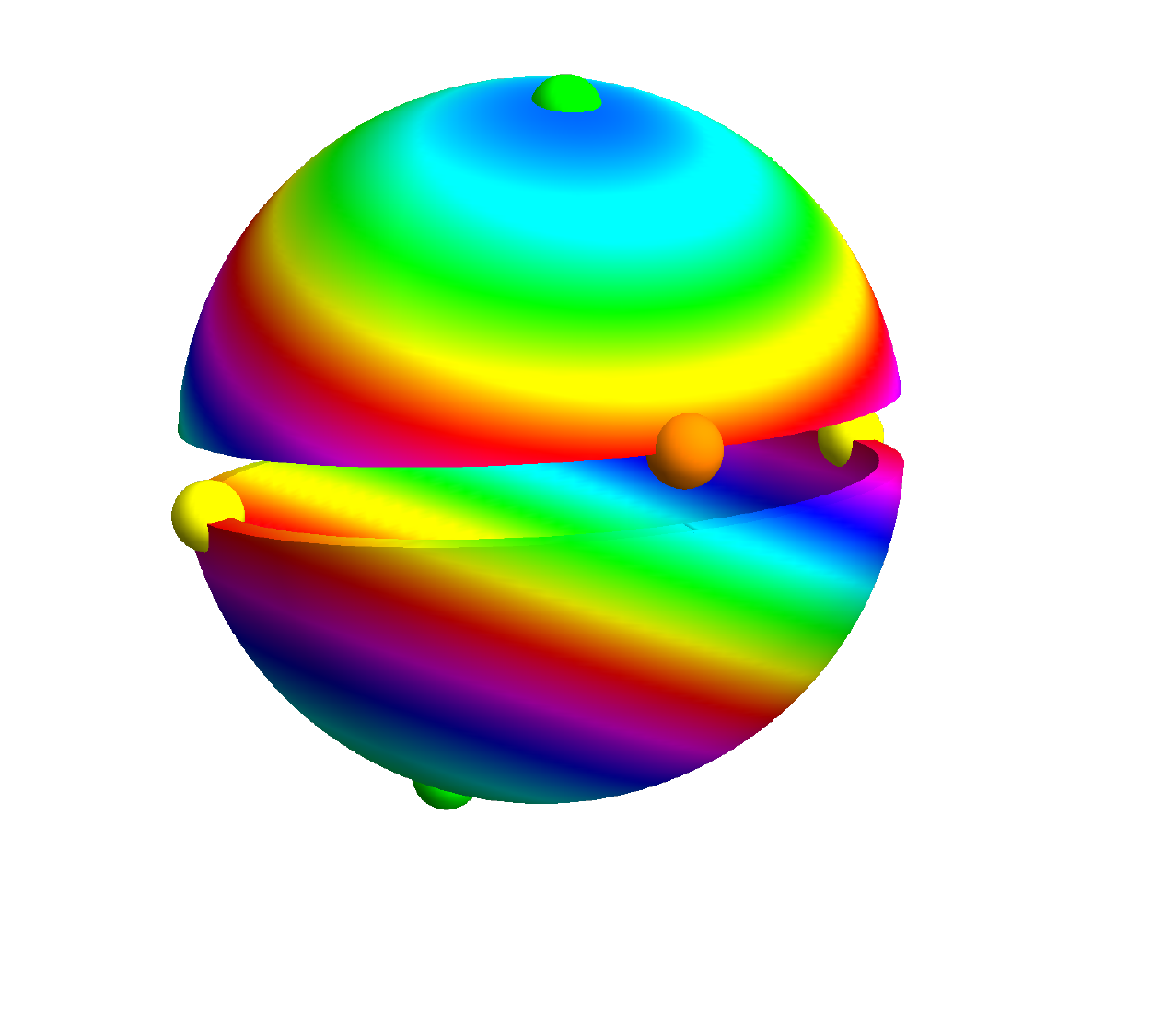}}
\label{sphere2}
\caption{
A sphere is orthotope as a union of two orthotope Riemannian manifolds with boundary
(balls). In this example, on each ball $M_j$, a function is given with three critical points
where two are on the boundary $\delta M_j$ circle. The indices of the interior 
indices are $1$, the indices on the boundary are $1$ and $-1$. The sum of all indices 
in each half sphere is $1$, the Euler characteristic of $M_j$. Once glued, the
boundary curvatures can be removed together with the boundary manifold. When doing gluing
along smooth hypersurfaces, we only need the probability spaces to be invariant under
the $f \to -f$ symmetry and not the full isotropy assumption.
}
\end{figure}

\section{Probability spaces}

\paragraph{}
Let $M$ be a Riemannian polyhedron embedded in an ambient linear Euclidean space $E$ 
and let $\Omega$ denote the set of all Morse functions on $M$ in the sense defined in 
the last section. By Sard's theorem, this space is rather large. 
For example, in the set of linear functions $f_a$ given a point $a$ in the unit sphere 
$\mathbb{S}$ of $E$ defined by $f_a: E \to \mathbb{R}$ ,there is an early result of Morse:

\begin{lemma}
Let $M$ be a Riemannian polyhedron, a Riemannian manifold with boundary. 
For almost all $a \in \mathbb{S}$, the function $f_a$ restricted to $M$ is Morse in the
sense defined in the last section. 
\end{lemma} 
\begin{proof}
For almost all $a$, the critical points of $f_a$ in the interior of $M$
are Morse: the gradient $g=df_a=a$ defines a smooth map $g: M \to E$. Almost all values $a \in E$ are
regular values for $g$ by the {\bf Sard theorem}. For a regular value, $dg(x)=H(x)$ has maximal rank.
In our case, where $M$ has boundary parts, 
we also want to assure additionally that $f$ induces Morse functions in the interior
of each boundary part $M_k$ of $M$. There are finitely many simplices and as the intersection 
of finitely many set of full measure have full measure:
let $H_k$ denote the Hessian of $f$ on $M_k$. 
Since each $\Omega_k = \{ f \in \Omega \; | \; \nabla f(x)=0, {\rm det}(H_k(f)(x))=0 \}$
has zero measure also the union has zero measure. 
\end{proof}

\paragraph{}
Let $M=r(T)$ be triangular surface in $E=\mathbb{R}^3$, where $T$ is a triangle in $\mathbb{R}^2$. 
The set $\Omega$ is now a circle. All except the three points $a_i$ for which $a_i$ is perpendicular
to one of the sides are Morse. Each function $f$ has just one critical point, the vertex where $f$ is
minimal and where $S_f(x)$ is empty. 

\paragraph{}
Let $M$ be a simply connected region in the plane with smooth boundary. The gradient of $f$ restricted to the interior
is never zero so that the curvature in the interior is zero and only Lagrange critical points contribute.
The curvature is $\alpha'(t)/(2\pi)$, if $r(t) = e^{i \alpha(t)} r(t)$ is the description of a parametrization with 
constant speed. Hopf's Umlaufsatz assures that the total curvature is $1$, which is the Euler characteristic of $M$. 
When looking at the boundary, then every function $f$ has at least two critical points, the maximum (with index $-1$)
and minimum with index $1$. A function is Morse if the critical point on the boundary. 

\paragraph{}
If the $2d$-manifold $M$ admits a global frame bundle and $\Omega$ is a probability space of Morse functions, we look at
the {\bf product probability space} $\Omega^d = \Omega \times \cdots \times \Omega$ and define there for any 
points $y_1, \dots, y_d \in M$ 
\begin{eqnarray}
 f(x_1,x_2, \cdots, x_{2d-1},x_{2d}) &=& f_1(x_1,x_2,y_{13}, \dots  y_{1d}) + f_2(y_{21},y_{22},x_3,x_4, \dots, y_{2d})  \\
                                     &+& f_d(y_{d1},y_{d2}, \dots, x_{2d-1},x_{2d})  \; .
\label{construction}
\end{eqnarray}
The coordinates at the point $x$ are given by the global section of the frame bundle given in $M$.
We can assume the coordinates to be {\bf normal} obtained by using the exponential map of a small disc in $T_x(M)$ to get
the grid lines. We can look at the function $f$ as a random variable on $\Omega \times M)^d$, where each 
$y_k \in M$ is an {\bf additional parameter} to keep the individual parts independent. 
Each of the summands is now a function of two variables only. 
Let us remark that construction~(\ref{construction}) produces from convex functions $f_i$ a new convex function $f$. 
For smooth functions one can see that from the fact that the Hessian matrix $d^2f$ is positive semi-definite if
each of the Hessians $d^2 f_k$ is positive semi-definite \cite{Rockafellar}. 

\paragraph{}
The indices of $f$ relate to indices restricted to $2$-dimensional planes which allows us to extract sectional curvature.

\begin{lemma}
If $x$ is a critical point of $f$, then each $(x_{2k},x_{2k+1})$ is a critical point of $f_k$ and 
$i_f(x) = i_{f_1}(x_1,x_2) \cdots i_{f_d}(x_{2d-1},x_{2d})$. 
\end{lemma}
\begin{proof}
The gradient of $f$ is zero means that the gradients of each of the summands is zero. The index is independent
of the coordinate system. As we have essentially a product situation, the indices multiply:
the number of negative eigenvalues of the Hessian add when taking a direct product. 
\end{proof} 

\section{Gauss-Bonnet}

\paragraph{}
Given a Riemannian polytop $M$ embedded in $E$, that is a compact Riemannian $2d$-manifold $M$ with boundary 
embedded in an Euclidean space $E$, we look at a space $\Omega$ of smooth functions on $E$ which 
induce Morse functions on $M$. Assume also that a probability measure $\mu$ on $\Omega$ is given. 

\paragraph{}
In our case, a natural probability measure is the normalized volume measure on the unit sphere $S$ of $E$,
where for each $a \in S$, the function is given by $f_a(x) = a \cdot x$.
The set of Morse functions in $\Omega$ has then full measure. Let $\mathcal{A}$ denote
the Borel $\sigma$-algebra on $\Omega$ and $dV$ the volume measure on $M$ normalized to be a 
probability measure and so a probability space $(\Omega,\mathcal{A},\mu)$. Poincar\'e-Hopf assures that
for $f \in \Omega$ that $\int_M i_f(x) dV = \chi(M)$, where $i_f(x)$ is the index measure
supported on finitely many points and $\chi(M)$ the Euler characteristic.  

\paragraph{}
A probability measure $\mu$ on $\Omega$ defines an {\bf index expectation curvature}
$K(x)=\int_{\Omega} i_f(x) d\mu$ which depending on the measure $\mu$ can also be
a generalized function, actually a measure. On a Riemannian polytop $M$ with boundary $\delta M$, 
we have the following prototype result.

\begin{thm}[Gauss-Bonnet]
$\chi(M) = \int_M K(x) dV$. We can split $K = k + \kappa$ into curvature $k$ in the interior
and curvature $d\kappa$ on the boundary $\delta M$.
\end{thm}

\begin{proof}
Fubini's theorem gives Gauss-Bonnet $\int_M K(x) dV = \chi(M)$. The splitting into interior
and boundary curvature is just notation. For isotropic measures $\mu$, the measure $k$ in the 
interior of $M$ is absolutely continuous.
\end{proof} 

\paragraph{}
Both $k$ and $\kappa$ depend on the probability space $(\Omega,\mathcal{A},\mu)$. 
This is very general. For this result, the manifold can also be a disjoint union 
of manifolds $M_j$ with boundary which can have different dimensions. 
As we will see on each of these manifolds $M_j$, a different probability space 
can be taken. We can therefore also write 
$$ \chi(M) = \sum_{j=0}^{2d} \int_{M_j} K_j(x) dV_j(x) \; ,  $$
where $K_j(x)$ is the curvature on the $k$-dimensional part $M_j$ of $M$
and $dV_j$ the volume measure on that open Riemannian $j$-dimensional manifold $M_j$. 
What we will have to establish is that after gluing the interior curvature can be
discarded. 

\paragraph{}
{\bf Remark 1):} The manifold $M$ could be replaced with a geometric realization of a finite abstract
simplicial complex, where each $k$-simplex is equipped with a Riemannian metric 
coming from a neighborhood of the complex in an embedding in $\mathbb{R}^k$. 
The complex $G$ does not have to be pure. 
The result even extends to {\bf chains}, which are finite linear combinations 
$c=\sum_i c_i \sigma_i$, where $\sigma_i$ are simplices and $c_i$ are real numbers. 
Curvature is then just defined as $K=\sum_i c_i K_i$, where $K_i$ are curvatures
on $\sigma_i$. The Euler characteristic of a chain is 
$\sum_i c_i \chi(\sigma_i) = \sum_i c_i$ as $\chi(\sigma_i)=1$. 

\paragraph{}
{\bf Remark 2):} The valuation property $\chi(A \cup B) = \chi(A)+\chi(B)-\chi(A \cap B)$ 
of Euler characteristic could also be extended to more general sets like $M \setminus G$,
where $G$ is a geometric realization of a finite abstract simplicial complex embedded in the interior of $M$. 
There would be a Gauss-Bonnet theorem for such sets. 

\paragraph{}
As in general, for independent random variables, we have

\begin{lemma}
The index expectation curvature of $(\Omega \times M)^d$ is the product of the
index expectation curvatures: $K = K_1 \cdots K_d$, where each of the $K_j, j=1, \dots, d$ is a
sectional curvature.
\end{lemma}
\begin{proof}
Independent random variables $X,Y$ are decorrelation, which is equivalent to
${\rm E}[X Y] = {\rm E}[X] {\rm E}[Y]$.
\end{proof}

\section{Orthotope Riemannian manifolds}

\paragraph{}
When looking at Morse theory at the boundary of a polyhedron, we will need to understand what
happens if several polyhedra $M_j$ are joined together. The {\bf gluing problem} is to understand the curvature
at such a point, if several polyhedra $M_j$ equipped with different probability spaces $\Omega_j$ leading to different
curvatures on $M_j$ are glued together. We are especially interested in the 
{\bf gluing curvature} at points which after gluing are in the interior of $\bigcup_j M_j$. 

\paragraph{}
It turns out that we do not even have to worry about any other compatibility
of the various probability spaces $\Omega_j$ if an isotropy condition and orthogonality conditions
is satisfied: the interior gluing curvature will disappear. 
For gluing along hypersurfaces without boundary, one can simplify the story in 
that the probability spaces $\Omega_j$ need only to be invariant under the involution $f \to -f$. 
In general more symmetry is needed. Terminology from convex analysis \cite{Rockafellar} is helpful here when introducing 
the new notion called {\bf ``orthotope"}.

\paragraph{}
Orthotope polytopes generalize {\bf hyper rectangles} in Euclidean space. The condition is
that for two vectors appearing in the $1$-skeleton of $M$ are either parallel or perpendicular. 
An {\bf orthotope Riemannian polytop} $M$ is a Riemannian polytop for which orthogonality 
conditions satisfied at the boundary. We express this now in a way which we will actually use. Given a
point $x$ at the boundary of $M$, we can look at the {\bf Fenchel cone} $C$ defined in $T_xM$ which consists
of the linear span of all the vectors in $T_xM$ which point into $M$. The {\bf dual Fenchel cone} $\hat{C}$ of
$C$ is defined as the set $\{ w \in T_xM,  v \cdot w \geq 0$ for all $v \in C \}$. An 
orthotope Riemannian polytop has the property that at every boundary point one has $C(x)=\hat{C}(x)$. 

\paragraph{}
Now, let us call a Riemannian manifold $M$ to be an {\bf orthotope Riemannian manifold} if it can be partitioned
into arbitrarily small orthotope Riemannian polyhedra. 

\begin{figure}[!htpb]
\scalebox{0.5}{\includegraphics{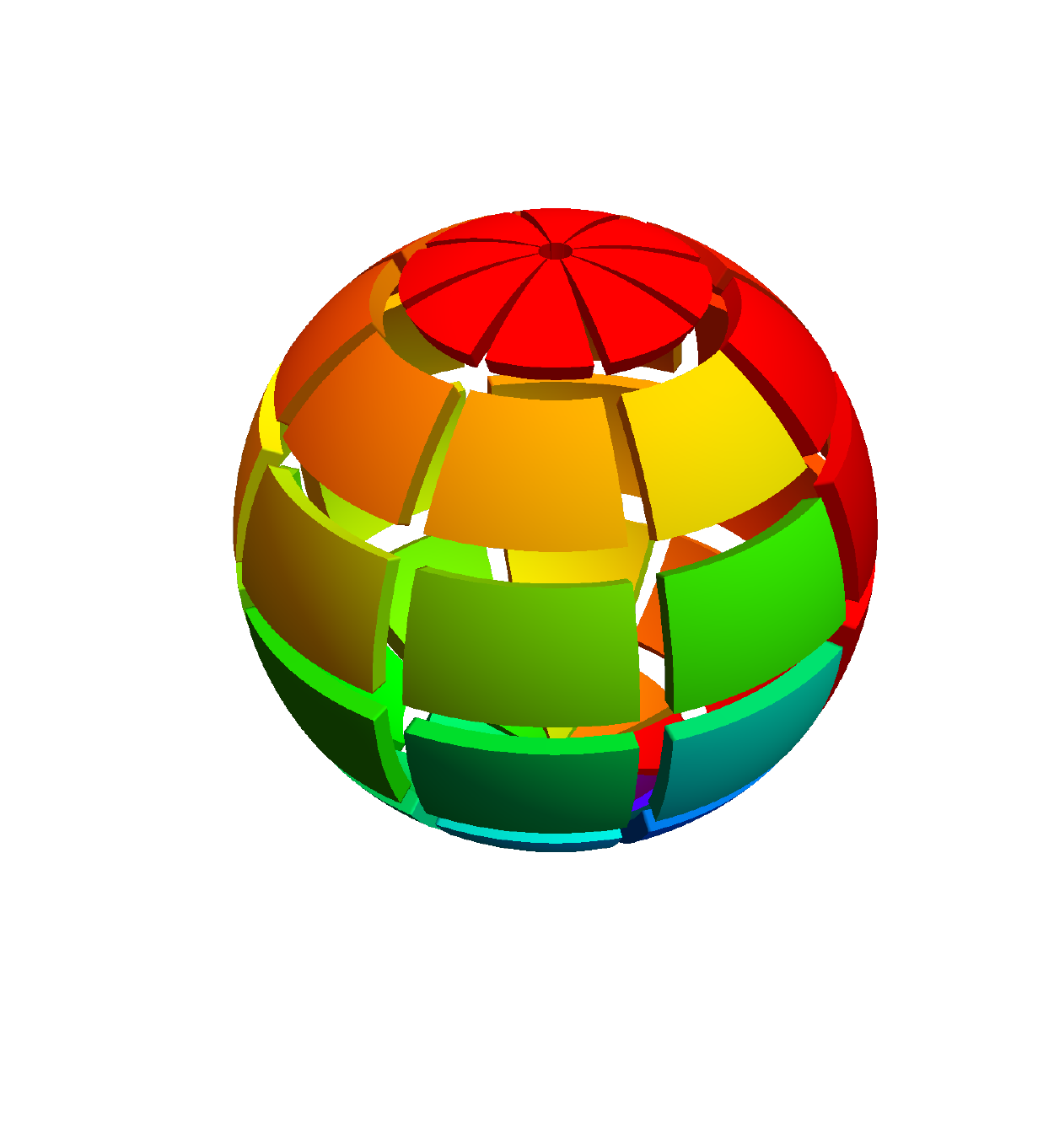}}
\scalebox{0.5}{\includegraphics{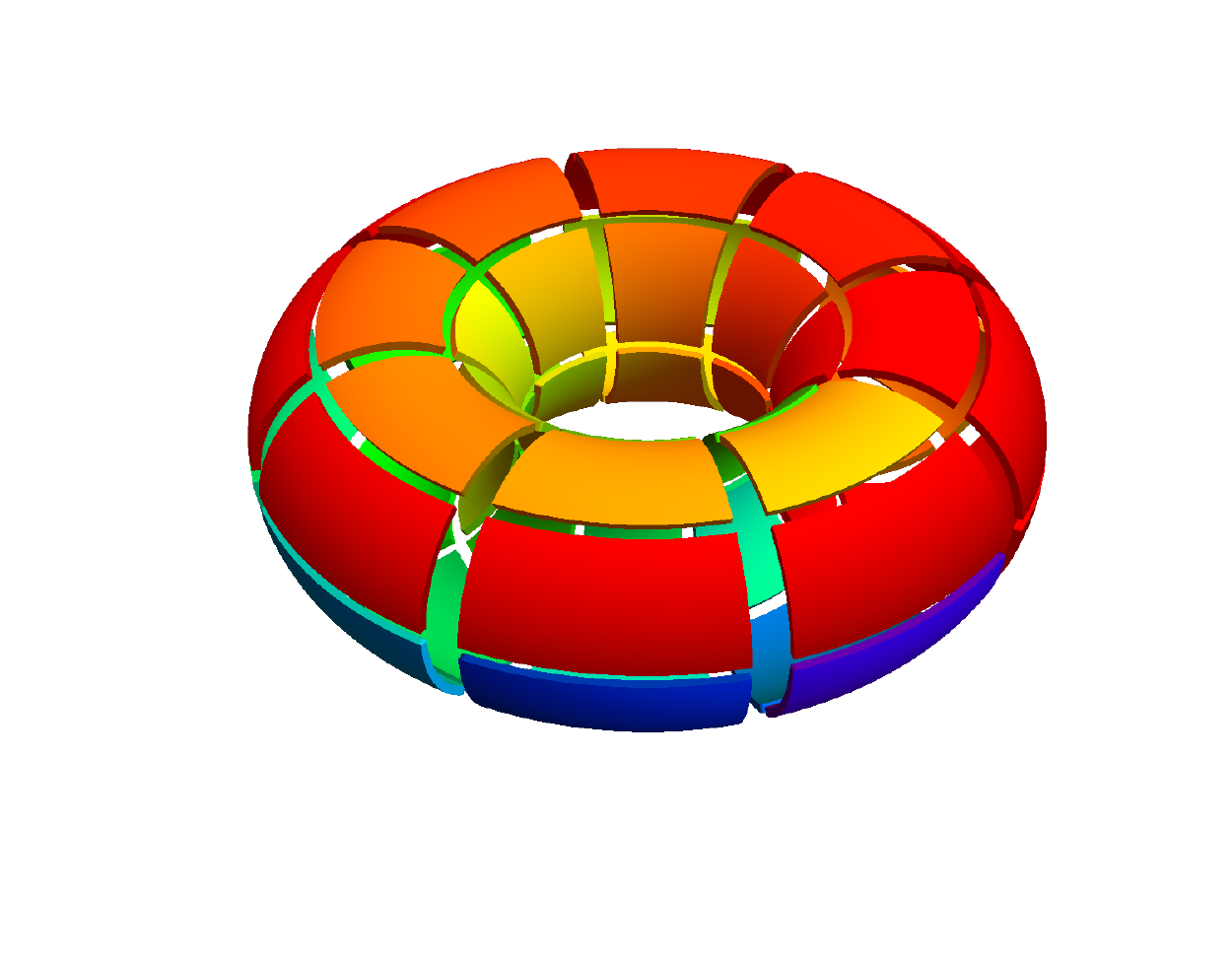}}
\label{orthotope}
\caption{
In any dimension, spheres and tori are orthotope Riemannian manifolds. 
}
\end{figure}

\paragraph{}
For example, the $2$-sphere $\mathbb{S}^2$ can be partitioned into $8$ Riemannian orthotopes which are
all triangles with $90$ degree angles. This generalizes to arbitrary spheres $\mathbb{S}^m$ by 
induction. Cut the sphere $M$ into two balls intersecting in a $d-1$ dimensional sphere $N$. 
Now build the orthogonal grid in $N$ which can also be seen as a graph. Build a suspension of
this and make sure that the new connections are perpendicular at each node. 
In general, we can say that if $M$ is an orthotope Riemannian $2d$-manifold, then 
the suspension of $M$ is orthotop.

\paragraph{}
Why do we need the duality condition?
If $x$ is a critical point of a Morse function $f$ 
in the interior of the $(2d-1)$-dimensional part $M_k$ of the boundary of 
a Riemannian polyhedron $M$, then $\nabla f(x)$ is perpendicular 
to the boundary $M$. The set $S_r(x) \cap M$ is a half sphere
which is a topological ball. So, if $\nabla f$ points outwards, 
then $S_r^f(x) = S_r(x) \cap \{y,  f(y) \leq f(x) \}$ 
is still a half ball and the point has index $i_f(x) = 1-\chi(S_r^f(x)) =0$. 
We see that in the case when a critical point is $(2d-1)$-dimensional, 
then we need the gradient to point {\bf into the dual Fenchel cone} defined at the 
point in order to have an index which has a chance to be non-zero. 

\paragraph{}
What happens if we have a narrow Fenchel cone $C$ is that the dual Fenchel cone $\hat{C}$ is large.
In the limiting case if the Fenchel cone is a half line, the dual Fenchel cone is a hyperplane.
On the other hand, for a wide Fenchel cone $C$, the dual Fenchel cone $\hat{C}$ is narrow. 
Again, in a limiting case, if the $C$ is a half space, then $\hat{C}$
a half line. What happens at a point $x$, where several orthotope $M_j$ meet 
is that a function $f$ can only have a non-zero index in one of the $M_j$. 
In general, there can be over-counting and since the curvature is positive at
the boundary of small Riemannian polyhedra $M_j$ (it is the same curvature we
have for $K_{GBC}$ when considered on polytopes), there can only be over-counting of 
curvature and not under-counting. This leads to the inequality $\gamma(M) \geq \chi(M)$. 

\paragraph{}
{\bf Example:} For a polyhedral convex set $M$, an intersection of 
finitely many closed half spaces in $\mathbb{R}^{2d}$,
then if $x \in M$, for sufficiently small $r$, the set $B_r(x) \cap M$ 
is either a ball or then the intersection of a ball with a {\bf convex cone}, 
(a closed subset which when $x$ is translated to the origin is closed under
addition and scalar multiplication). 
One reason why it is nice to take Morse functions $f$ on $E$ 
which are affine is that they are {\bf convex functions} in the sense 
that the epigraph $\{ (x,c) \in E \times \mathbb{R}, x \in E, c \geq f(x) \}$
is a convex set. It turns out that given convex functions $f_1(x,y),f_2(x,y)$ 
also newly built functions like 
$f(x,y) = f_1(x,v) + f_2(u,y)$ parametrized by $u,v$ are convex functions. 
The new probability space of functions we are going to construct 
consist of such functions. 

\paragraph{}
Let us look at the case when $M$ is a polyhedron in 
$E=\mathbb{R}^2$ and $f: \mathbb{R}^2 \to \mathbb{R}$
is a linear function $f(x) = a \cdot x$ with a vector $|a|=1$. 
For almost all $a$, such a function $f_a$ is Morse and the critical 
points are on vertices of the polyhedron. 

\begin{figure}[!htpb]
\scalebox{0.3}{\includegraphics{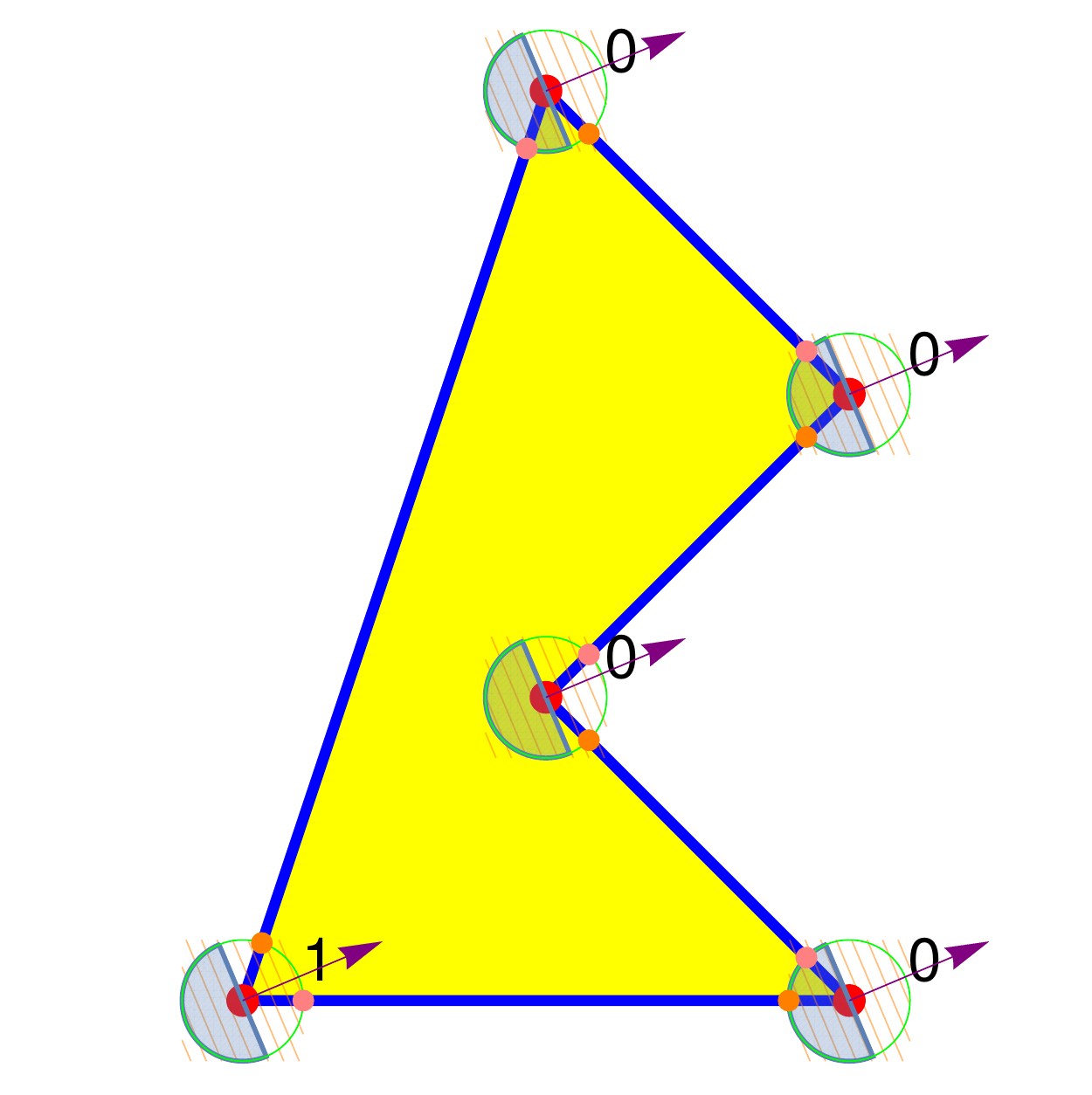}}
\scalebox{0.3}{\includegraphics{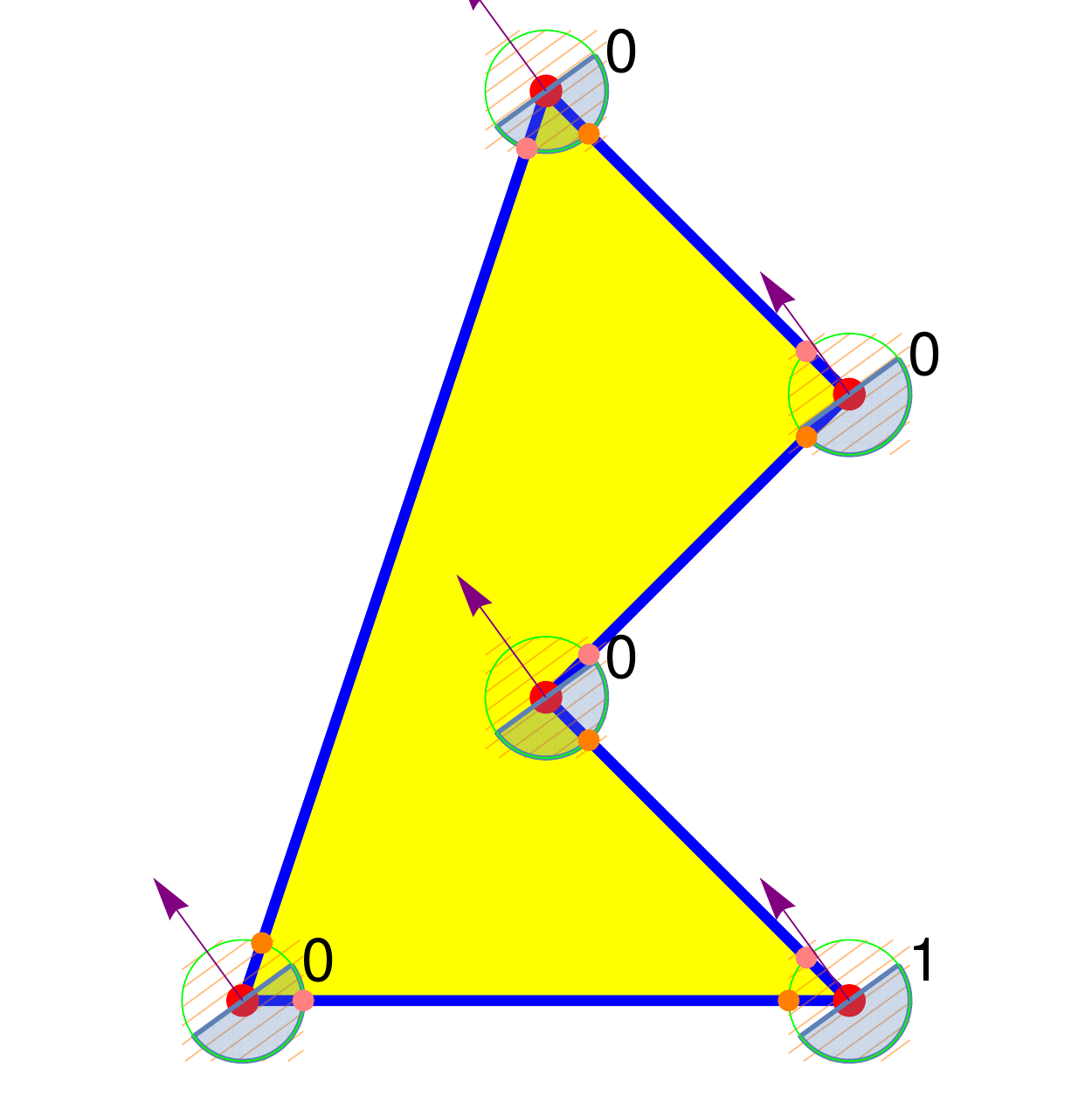}}
\scalebox{0.3}{\includegraphics{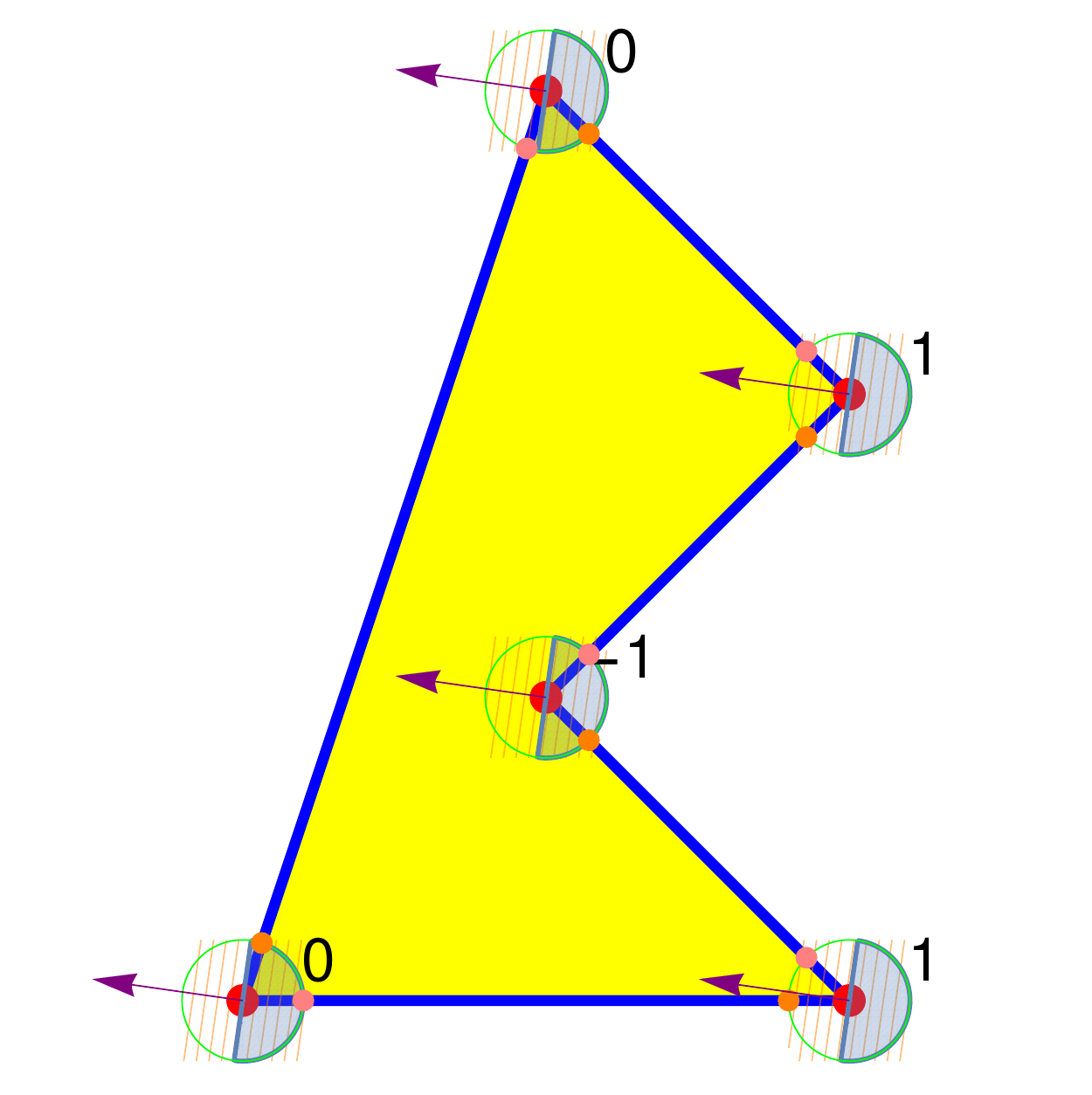}}
\label{polygon}
\caption{
Three cases of a linear Morse function $f_a$ on a polyhedron $M$. 
The index $i_f(x)=1-\chi(S_f(x))$ 
can only be non-zero if $\nabla f$ points into the dual
cone of $B_r(x) \cap M$. Poincar\'e-Hopf assures 
$\sum_x i_f(x)=1= \chi(M)$. When averaging over all $a$, we
get the Hopf-Umlaufsatz with curvature supported on finitely many 
points. In this case, there is one point with negative curvature. 
The isotropy condition of probability spaces we will impose implies
that the curvature at a vertex will be a solid angle curvature, the
same which appears for Gauss-Bonnet-Chern. This is the reason why
our new curvature $K$ survives the {\bf Allendoerfer-Weil gluing} in the same
way than $K_{GBC}$. 
}
\end{figure}

\section{Gluing orthotope polyhedra}

\paragraph{}
Assume, the $2d$-manifold $M$ is Nash embedded in an ambient Euclidean space $E$. Assume that
$M$ is divided into two parts $M^{\pm}$ and that we have two different probability spaces 
$\Omega^+$ and $\Omega^-$ of Morse functions $M^{\pm}$. We assume that the probability spaces are
invariant under rotations $f \to f(A)$ for $A \in SO(E)$ so that $f \to A(f)$ are 
measure preserving involutions both on $\Omega^+$ and $\Omega^-$. We call this a {\bf isotropy condition}. 
Let $M^0$ be the intersection $M^+$ with $M^-$. Applying the Gauss-Bonnet result on both sides
gives curvatures $K^{\pm}$ in the interior of $M^{\pm}$, as well as boundary curvatures
$\kappa^{\pm}$ restricted to $M^0$. Any convex combination $p_1 \kappa^+ + p_2 \kappa^-$ 
with $p_1+p_2=1$ produces now a curvature $\kappa^0$ on $M_0$ which integrates up to $\chi(M^0)$.
Any critical point on $M^0$ comes from a critical point of one of the two sides. 
Every function on $M^+$ and $M^-$ produces a Morse function on $M_0$.
Index expectation there gives $\chi(M_0)$. 

\paragraph{}
The valuation property of Euler characteristic means $\chi(M) = \chi(M_1) + \chi(M_2) - \chi(M_0)$.
It follows that we can forget the measures $\kappa^+, \kappa^-$ as well as the interior of $M_0$ 
when gluing two manifolds $M^+$ and $M^-$ together. This is the simplest case where we glue along
a $2d-1$ dimensional simplicial complex. If we glue such that a $0$-dimensional part of $M$ gets
berried inside $M$, then we have to make sure that the sum of the probabilities adds up to $1$ there
so that this gluing measure can be removed. 

\paragraph{}
If the gradient of a Morse function $f$ at a boundary point $x$ of dimension $(2d-1)$ of a 
$2d$-dimensional Riemannian polyhedron $M$ points outwards, then $i_f(x)=0$. 
Proof: Technically, this follows from the fact that the dual cone of the part on one side of the 
hyperplane $\Sigma$ agrees with the cone itself. We use the fact that for every boundary point $x$ of the polyhedron $M$, 
the sub-manifold $S_r(x)$ is a topological ball. No, if $\nabla f$ points outwards,
then $S_f^-(x) = S_r(x)$ and because $\chi(S_r(x))=1$, we have $i_f(x) = 1-\chi(S_r(x)) = 1-1=0$. 

\begin{figure}[!htpb]
\scalebox{0.5}{\includegraphics{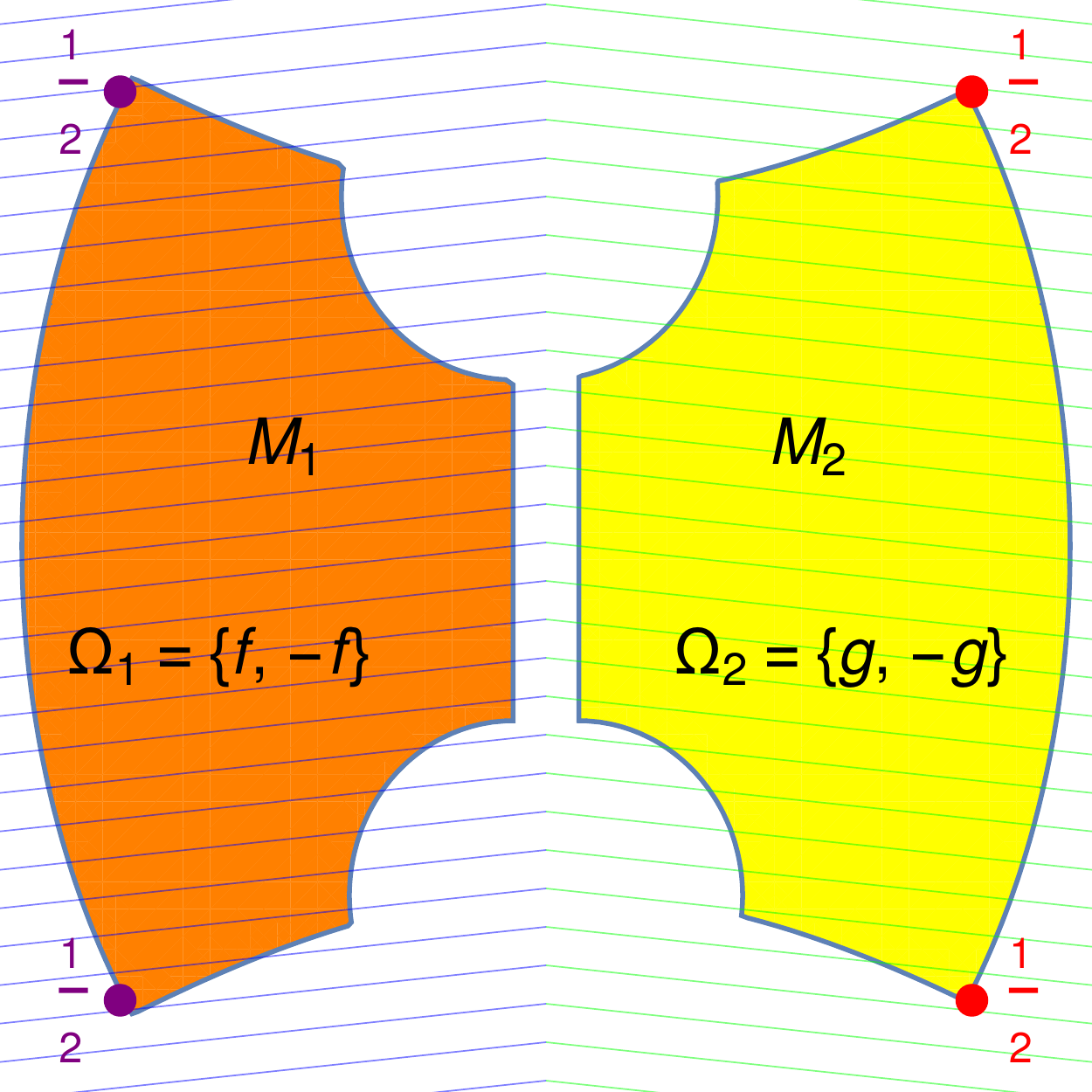}}
\scalebox{0.5}{\includegraphics{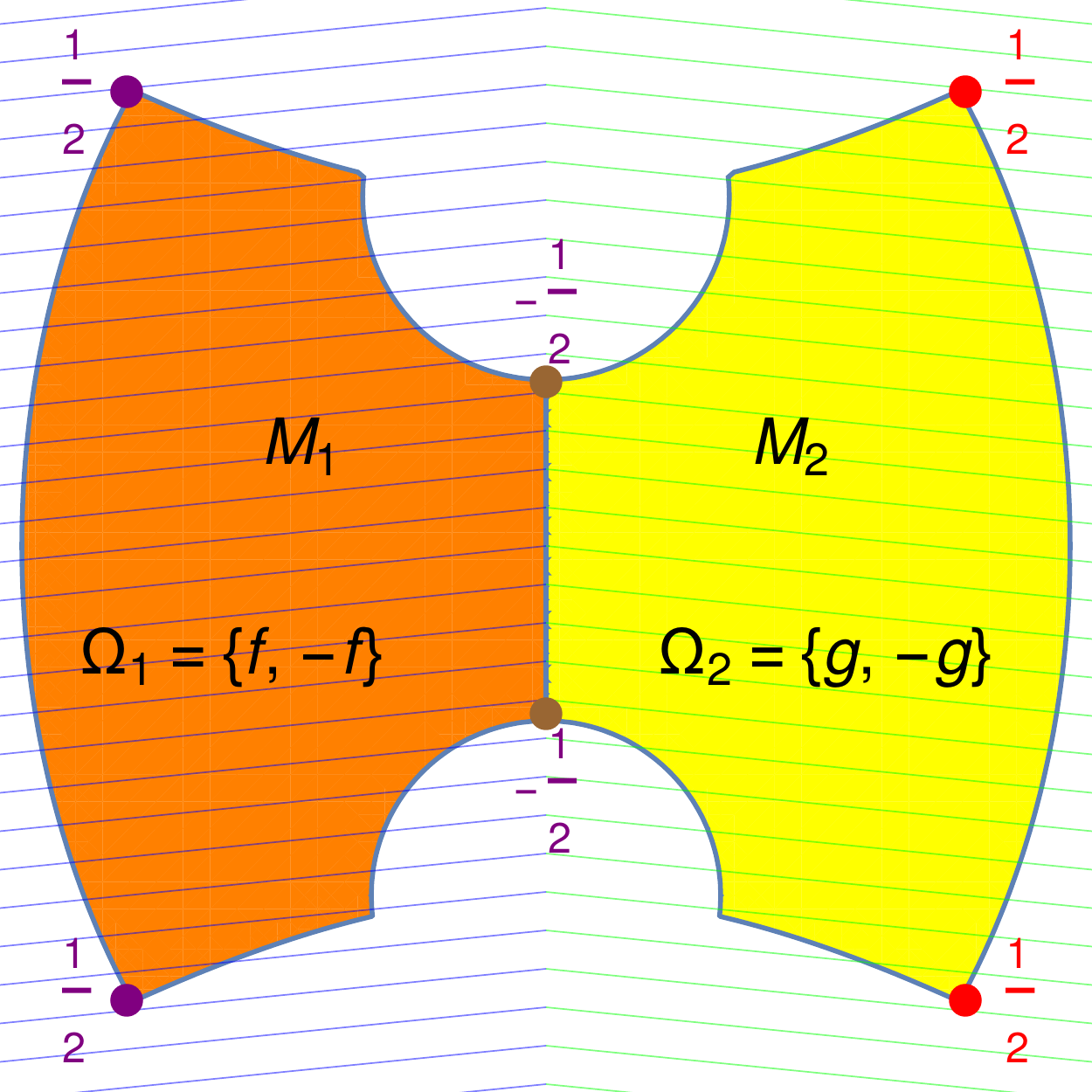}}
\label{orthotope2}
\caption{
Gluing orthotope manifolds. We see an example where $(M_1,\Omega_1=\{f,-f\})$ 
and $(M_2,\Omega_2=\{g,-g\})$ have no curvature on the intersection $M_1 \cap M_2$ but where 
the induced probability space $\Omega_0$ on the glued orthotope 
$(M_0,\Omega_0)=(M_1 \cap M_2,\{ f,-f,g,-g\})$ has curvature. The negative of this curvature is
what is present on $M_0 \subset M$. The total curvature of $M$ is still $1 = \chi(M)$. 
It is important to have orthotope conditions. It is also important to have a homogeneity 
condition in general. 
}
\end{figure}

\paragraph{}
So, the only critical points which matter are the critical points $x$, where $\nabla f(x)$
points {\bf into} the dual cone of the polyhedron which is in the orthotope case into the
polyhedron. This is where the orthotope condition comes in.
Note that this only holds under the milder parity condition of symmetry $f \to -f$ 
if the point $x$ is in the interior of a $(2d-1)$-dimensional boundary face of the 
Riemannian polyhedron $M$. If $x$ is in a smaller-dimensional points, the dual cone is in general different
from the cone and it is there, were the orthotope condition is needed. 
For a triangle $M$ for example (a smooth image $r(\Delta_2)$ of a 2-simplex), 
the dual cone of a vertex (a $0$-dimensional point in the Riemannian 
polyhedron $M$), only agrees with the cone if the angle at $x$ is equal to $\pi/2$, 
meaning that the triangle is orthotope. An example of an orthotope triangle 
is the $90-90-90$ triangle on the sphere. 

\paragraph{}
If a Riemannian polyhedron $M$ is a union of two polyhedra $M_k$ with smooth $(2d-1)$-dimensional boundary
and each is equipped with a Morse functions $f_k,-f_k$, then one of them $f_k$ induces a Morse function 
on $M_0$ with total index sum $\chi(M_0)$. The probability space $\Omega=\{ f_1,-f_1,f_2,-f2 \}$ 
equipped with any measure already produces a probability space of Morse functions
on $M_0$. The index expectation is $\chi(M_0)$. Removing all these critical points reduces the Euler characteristic 
to $\chi(M) = \chi(M_1) + \chi(M_2) - \chi(M_0)$. This gluing process allows to decouple the probability space 
on one part of the manifold from an other part. But this has only worked if we glue along one simplex. When building
up $M$, we have also to glue in pieces which close up some interior point, meaning that we have to glue along
smaller dimensional parts. 

\paragraph{}
Here is a gluing procedure we explored earlier
Assume $M$ is manifold which is the union of two orthotope manifolds $M_j$ and assume that
$M_0=M_1 \cap M_2$ is orthotop. If $(\Omega_i,\mathcal{A}_i,\mu_i)$ are 
any probability spaces of Morse functions on $M_j$ such that the index expectation curvature
is isotrophic at every vertex point in $M_k$, then 
$\gamma(M)=\gamma(M_1) + \gamma(M_2) - \chi(M_0)$. 
The reason is that for each critical point of $f_j$, exactly one of the Morse functions $f_j$ on $M_j$ contributes to a
critical point in $M_0$.  Due to the orthotope condition, there is no overlap of probability spaces and the two 
spaces induce a probability space $(\Omega_0,\mathcal{A}_0,\mu_0)$ in the intersection. 

\paragraph{}
The gluing works almost unconditionally 
if the intersection $M_0$ is a manifold without boundary. In that
case we only need that $(\Omega_i,\mathcal{A}_i,\mu_i)$ admit a measure preserving parity action $f \to -f$ 
for $i=1,2$. Now, we can produce a probability space $(\Omega_0,\mathcal{A}_0,\mu_0)$  of Morse
functions on $M_0$ by taking the push forward measures $\mu_{i,0}$ of the restriction map $M_i \to M_0$
and define $\mu_0 = 1/2(\mu_1 + \mu_2)$ which by symmetry $f \to -f$ is a probability measure. 
The orthotope condition is automatically satisfied at a point of $M_0$ also if $M_i$ are manifolds with 
boundary and $M_1 \cup M_2$ is a manifold.

\obs{
If $M=M_1 \cup M_2$ is a Riemannian manifold divided up along a manifold $M_0=M_1 \cap M_2$ and
$(\Omega_i,\mathcal{A}_i,\mu_i)$ is an arbitrary choice of parity symmetric probability spaces 
of Morse functions on $M_i$, then $\chi(M_1) + \chi(M_2) - \chi(M_0) = \chi(M_1 \cup M_2) = \int_M K(x) \; dV(x)$ where
$K(x)$ is the curvature $K_i$ in the interior of each $M_i$. No residue glue curvature is present
in $M_1 \cap M_2$.
}

\begin{figure}[!htpb]
\scalebox{0.4}{\includegraphics{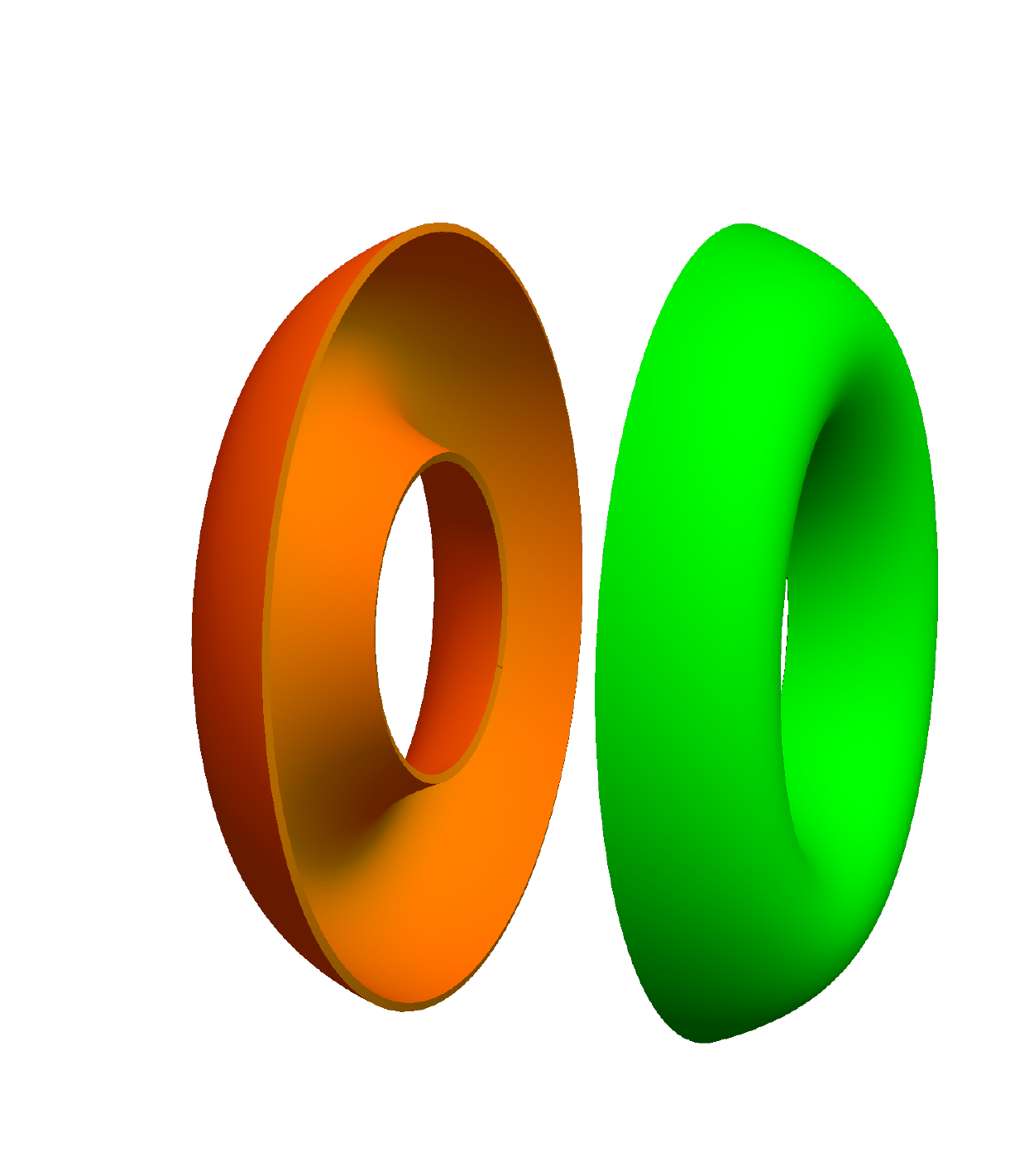}}
\label{toruscut}
\caption{
Cutting a manifold along a manifold allows to put completely different
probability measures of Morse functions on each part as long as the 
probability measures satisfy some parity condition. Gauss-Bonnet still
holds. It is when cutting things up further that we have to have some
isotropy and orthogonality condition to make things compatible at the
corners of the polyhedra. In the current situation, the curvature which 
is present on each boundary side after cutting disappears after gluing.
There is no residue glue. 
}
\end{figure}

\paragraph{}
The gluing also works nicely if the probability spaces $\Omega_i$ come from a global isotropic
probability space $\Omega$ like all linear functions in an ambient Euclidean space. In that case
the induced curvature is the Gauss-Bonnet-Chern integrand and the gluing
$\chi(M_1) + \chi(M_2) - \chi(M_0) = \chi(M_1 \cup M_2)$ still works well. This is the Gauss-Bonnet-Chern
theorem. But as Goroch has shown for positive curvature manifolds $K_{GBC}$ can become negative
at some points in dimension $6$ or higher. 

\paragraph{}
We originally thought that we can force the isotropic condition to hold everywhere.
But averaging over all frames at one point in $M_k$. While this makes the
curvature $K(x)$ isotropic at this point, it does not make it isotropic on all of the polytop $M_k$.
When the chamber $M_k$ gets small, also the error gets small and goes to zero if the diameter
of $M_k$ goes to zero, what happens is that we have also more and more boundary appearing. 
The discrepancy of isotropy everywhere is not small enough to go to zero, when adding things up. 
The above gluing procedure however works if we have a curvature which is isotropic
everywhere. This is where Weyl's invariance theory \cite{Weyl1946,Bott75} comes in and where 
Singer has conjectured that such a $K$ must be $K_{GBC}$ up to some divergence. 
And this is what {\bf Gilkey} proved in 1975 \cite{Gilkey1975}.
The curvature $K(x)$ integrating up to $\gamma$ only is {\bf coordinate independent} at every
point, but it is not {\bf metric independent}. This manifests here in {\bf gluing curvature}
at the boundaries of $M_k$ which are metric independent. In physics one would associate this
with mass or energy transfer between the different cells $M_k$. The analogy is because mass 
influences the metric.

\begin{figure}[!htpb]
\scalebox{0.3}{\includegraphics{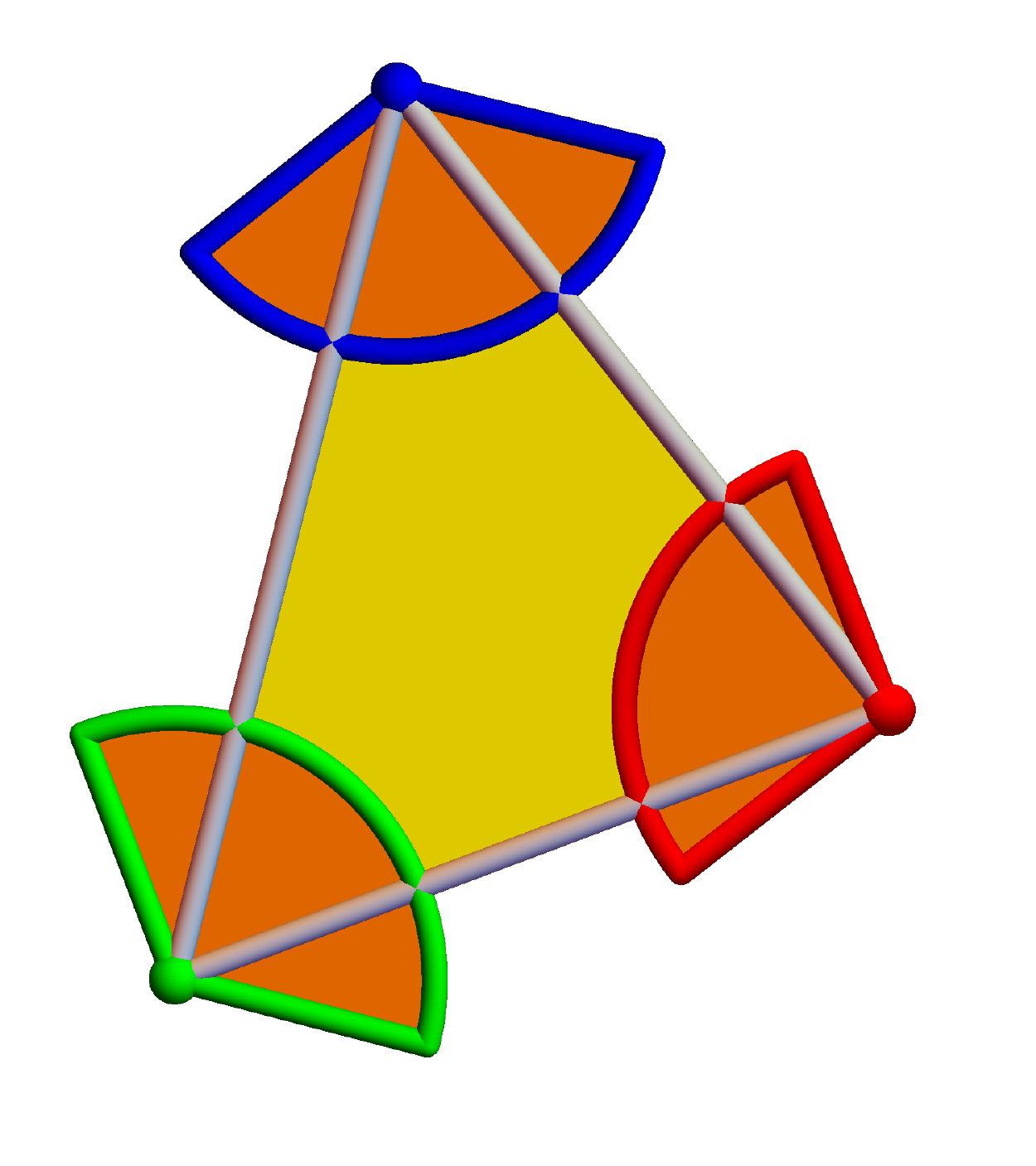}}
\label{cone}
\caption{
For a polygon, the curvature $K(x)$ a vertex is the angle of the dual 
Fenchel cone at the point $x$. The sum of the curvatures is $2\pi$. 
For a triangle with angles $\alpha,\beta,\gamma$ adding up 
to $\alpha+\beta+\gamma=\pi$, this is 
$\alpha'=\pi-\alpha, \beta'=\pi-\beta,\gamma'=\pi-\gamma$.
Gauss-Bonnet assures $\alpha'+\beta'+\gamma'=3\pi-\pi=2\pi$. 
}
\end{figure}

\begin{figure}[!htpb]
\scalebox{0.2}{\includegraphics{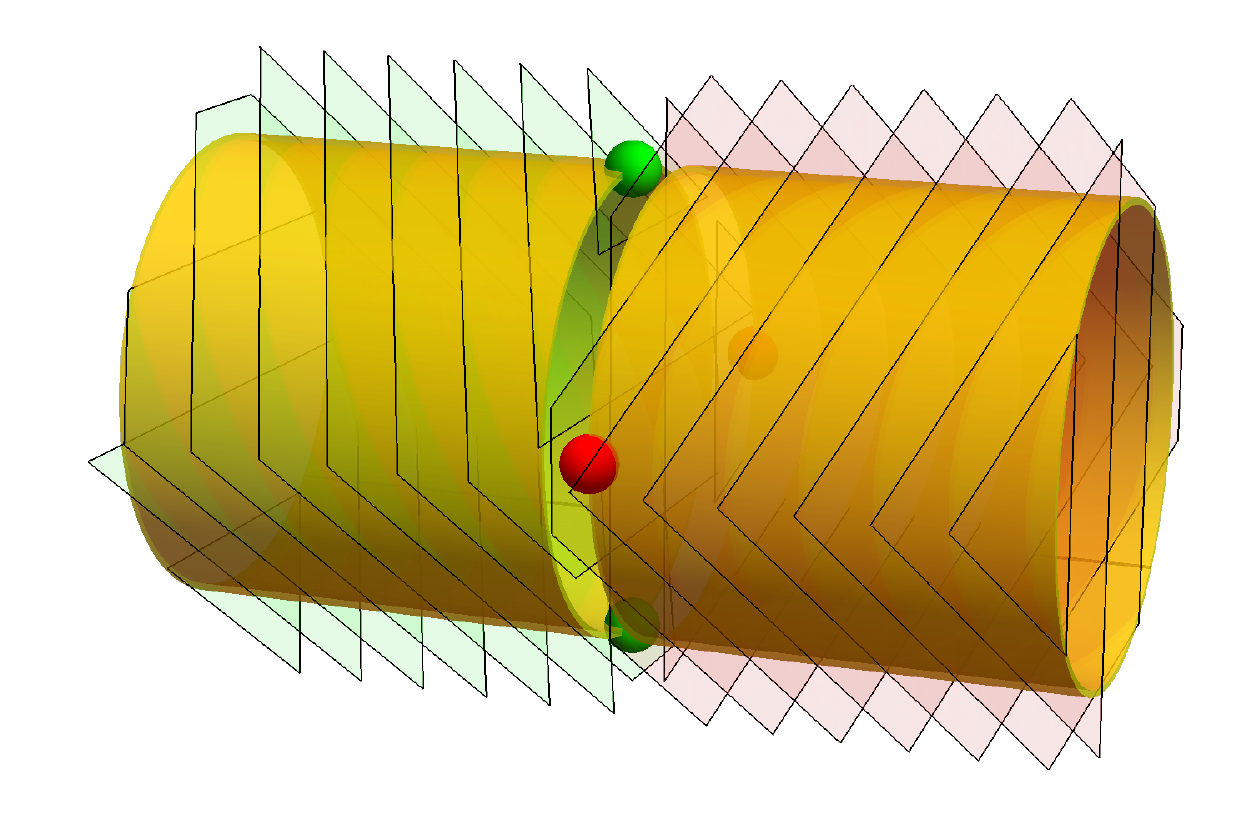}}
\scalebox{0.2}{\includegraphics{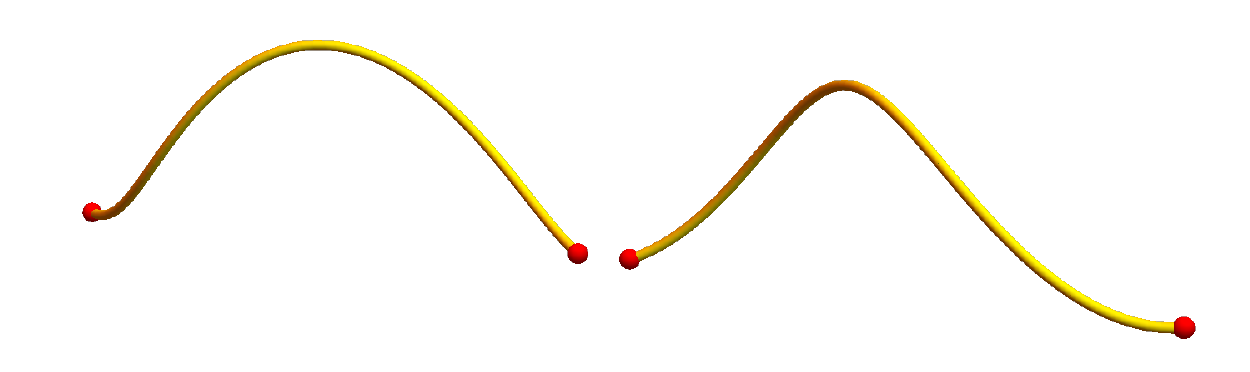}}
\label{glueing1}
\caption{
A situation where two orthotope manifolds $M_1,M_2$ are glued together. We can have different
probability spaces on $M_i$. Poincar\'e-Hopf and Gauss-Bonnet even works
for odd-dimensional manifolds, in the later case, the curvature is located
on the boundary. If isotropy holds at each end, then the
curvature of a line segment is $1/2$ on each end. When gluing them, the
interior curvature disappears and again just two end curvatures of value $1/2$ remain. 
}
\end{figure}

\begin{figure}[!htpb]
\scalebox{0.2}{\includegraphics{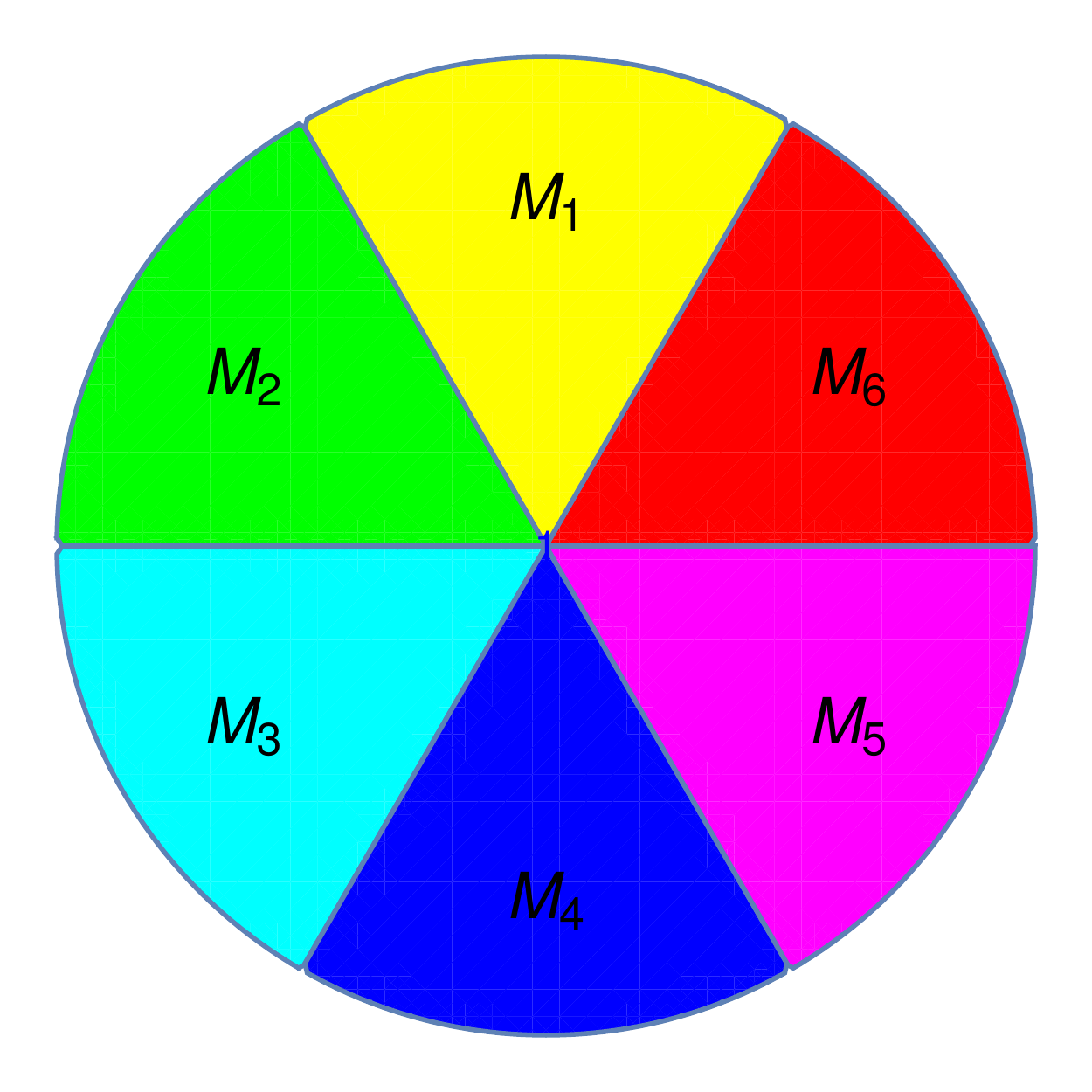}}
\scalebox{0.2}{\includegraphics{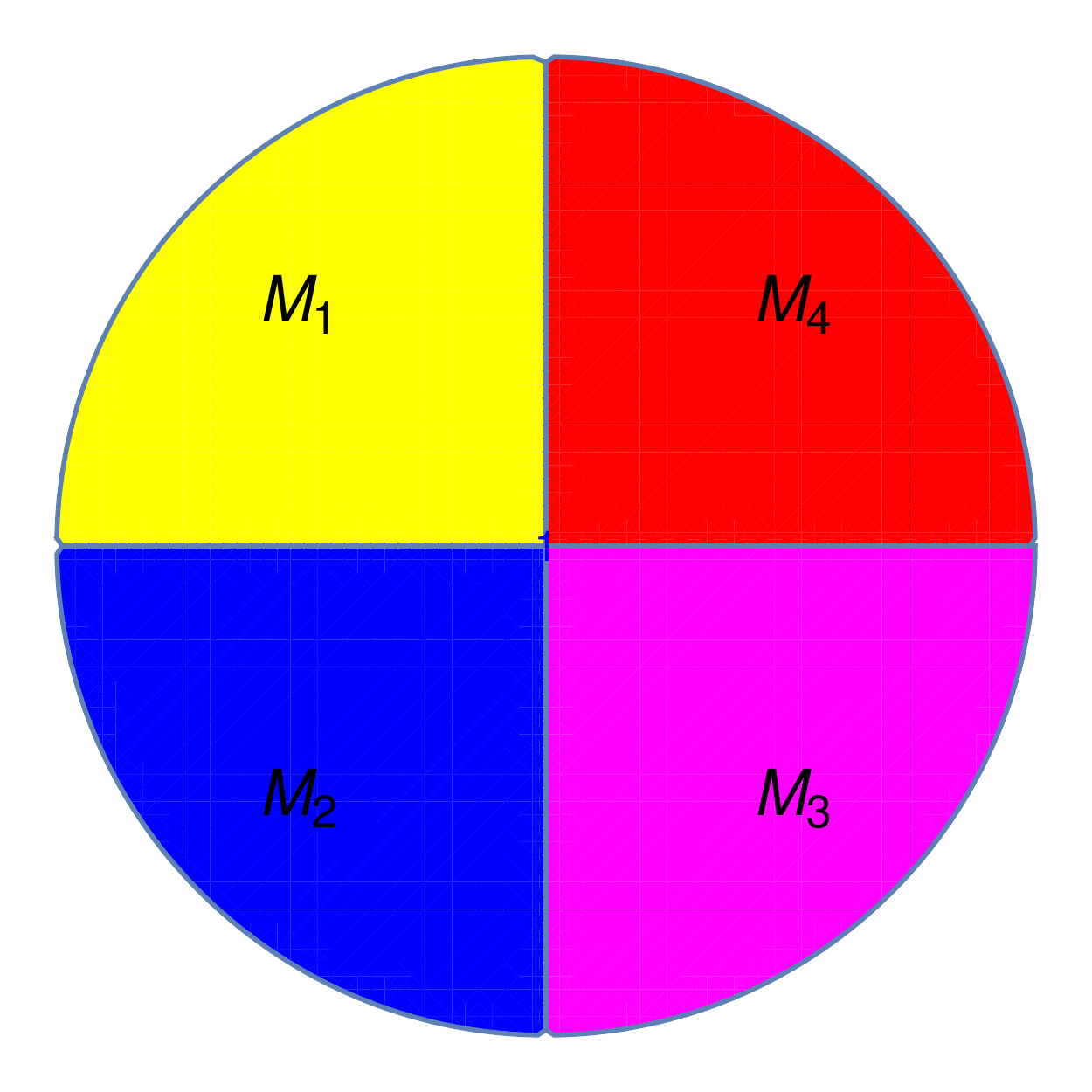}}
\label{glueing2}
\caption{
For a non-orthotope decomposition we can have probability spaces
on each part such that some positive curvature remains at the center.
This is the mechanism which can produce $\gamma_d(M) \neq \chi(M)$ because
some ``glue curvature" can remain in the interior. Also in the orthotope
partition case (displayed to the right), some gluing curvature
can remain in the interior gluing points if we lack isotropy at the 
boundary. 
}
\end{figure}

\section{Remarks and questions}

\paragraph{}
The most obvious question is to learn more about the nature of $\delta_d(M) = \gamma_d(M)-\chi(M)$
or $\delta(M)=\gamma(M)-\chi(M)$. 
As $\delta$ depends on the metric and frame on $M$, one can ask which metrics and frames
extremize $\gamma$. This question was interesting in the case of the
Hilbert action (which is metric dependent but not frame dependent). The simplest case is for $4$-manifolds, where
$$ K_d(x) = (32 \pi^2)^{-1} \sum_\sigma K_{\sigma(1),\sigma(2)} K_{\sigma(3),\sigma(4)}  \;   $$
is a sum of $4!=24$ permutations $\sigma$. One would like of course to have a
local condition which when satisfies makes $\int_M K(x) \; dV(x)$ locally maximal or locally minimal. 
The torus example metrics on $(\mathbb{T}^4,g)$ given in the example shows that one can not expect
to have global maxima or minima of $\gamma$ in general. 
Which manifolds satisfy $\gamma(M,g)=\chi(M)$ for some choice of frame? 
We do not even know yet a metric on $SU(3)$ and frame for which 
$\gamma(M,g) = \chi(M)=0$. Computing $\gamma_d(M)$ is hard in general if there are no large
symmetry groups on $M$. For $\gamma(M)$ it is even harder. 

\paragraph{}
As for negative manifolds, where the Cartan-Hadamard theorem holds, it would be interesting
to know whether $\gamma(M)$ or $\gamma_d(M)$ can be computed as a limit on the universal cover. 
While one has $\gamma(M)(-1)^d >0$, the Hopf conjecture states that $\chi(M) (-1)^d >0$
in that case. In light of that, it would be interesting to know of examples of negative curvature
manifolds, where one can compute $\gamma_d(M)$ of $\gamma(M)$ explicitly. 

\paragraph{}
We first thought to associate $\gamma$ with some kind of 
{\bf Dehn invariant} as it appeared that the property of being able to cut $M$ 
into orthotope polyhedra seem to matter for whether $\gamma_d=\chi$.  But that is not the case.
The story of Dehn is that for polyhedra in $\mathbb{R}^3$ a quantity like $\sum_k l_k \theta_k$ with length
$l_k$ and dihedral angle $\theta_k$ serves as a discrete analog of mean curvature.
As we thought of $\gamma(M)-\chi(M)$ to be associated on how $M$ can be partitioned into polyhedra $M_j$,
the analogy popped up. This analogy has faded, as there are examples like $\mathbb{T}^4$ which can perfectly 
well be chopped into orthotope polytopes, and on which metrics $g$ exist for which $\gamma(M)-\chi(M)$ can
take any given real value. 

\paragraph{}
The Dehn association also appeared appealing because of dimension.
Hilbert's third problem asked in 1900 whether two polyhedra of equal volume can be cut into polyhedral pieces
in such a way that one reassemble the second from the first. In dimension $2$, 
such a decomposition was possible, while in the three dimensional case it was no
more possible as Max Dehn, a student of Hilbert, proved first. Still, a Riemannian manifold analogue is
the question whether it is possible to cut two Riemannian manifolds into bricks with orthogonal boundary
parts. If different metrics are given, the question is to 
produce diffeomorphisms between the bricks which preserve the orthogonality condition at the 
boundary. Actually, at the moment, we do not know yet of any Riemannian manifold, where the orthotope
condition fails. All $2$-manifolds are orthotope, $3$-manifolds allow a triangulation leading to a Heegard splitting along 
an orientable genus g-surface. This suggests that any 3-manifold is orthotope but we do also not know 
this yet. 

\paragraph{}
The functional $\gamma_d$ or $\gamma$  is at first only defined for even dimensional manifolds. It can be extended to odd-dimensional
$(2d+1)$-manifolds by using the same formula and summing just over all groups of $d$ perpendicular
$2$-planes which can be built in $M$. The constant could be adapted in the odd-dimensional case too. One
suggestion is to take the constant so that the spheres $S^{2d+1}$ have $\gamma(M)=2$ which is larger 
than $\chi(M)=0$. In the case $d=3$, the function is proportional to scalar curvature because it averages all sectional
curvatures at a point. 

\paragraph{}
The curvature $K(x)$ was motivated by the Gauss-Bonnet theorem for {\bf Riemannian polyhedra} involving
index expectation using product probability spaces produced in a particular frame $t \in O(M)$ and
constructing positive curvature $K(x,t)$ on the orthonormal frame bundle $O(M)$
of $M$. It is the product of the sectional curvatures of $d$ pairwise perpendicular $2$-planes
$t_1,t_2, \dots, t_d$ obtained from the orthonormal frame $t=(t_1, \dots, t_d)$ in the tangent space
$T_xM$. The curvature $K(x)$ is then an average of $K(x,t)$ over all possible
orthogonal coordinate systems $t$ in $T_xM$ using the Haar measure. The curvature $K(x)$ differs
in general from $K_{GBC}(x)$ which does not have to be positive in the positive curvature case
\cite{Geroch,Klembeck}.  The curvature $K_{GGC}(x)$ is also an expectation of indices 
$i_f(x) = i_{t_1,f} \cdots i_{t_d,f}$ but where the same Morse function $f$ is used simultaneously 
for all $d$ two-dimensional planes. That expectation is $t$-independent if the probability 
space of the Morse functions $f$ is rotational and translational invariant in an ambient Euclidean space.

\paragraph{}
The classical Gauss-Bonnet-Chern curvature $K_{GBC}$ can be interpreted integral geometrically if 
we see $(-1)^{\sigma} R_{\sigma_i,\sigma_j,\sigma_k,\sigma_l}$ 
as the index expectation of functions $f(e_{\sigma_i},e_{\sigma_j})$ restricted to the plane spanned by 
$e_{\sigma_k}$ and $e_{\sigma_l}$. Now, a function $f$ in $E$ can be written as $f(x_1,\cdots, x_{2d})$ 
which is a combination of functions only depending on $_i,x_j$. 
And a similar argument than the lemma in the introduction shows that the integral expectation over
all functions can be written as an average over functions on coordinate planes. 

\paragraph{}
Integral geometry sees the entry $R_{ijkl}$ of the curvature tensor
is a {\bf conditional expectation} of Morse functions restricted to the $ij$ slice, where
the expectation is taken over Morse functions when integrated over all except the $k,l$ variables is done.
Assume $M$ is embedded in an ambient space $E$. In an orthonormal coordinate system,
$R_{ijkl}(x)$ is the index expectation of the probability space $\Omega$ of linear functions
$f_a$ of functions with in $\{ x_m=0, m \neq i,j \}$
on the surface $\exp(\Sigma_{ij})$ where $\Sigma_{ij}$ is the plane in $T_xM$ spanned
by $e_i,e_j$.

\paragraph{}
The classical Gauss-Bonnet-Chern theorem can be proven similarly to what we did here
following footsteps of earlier work by Fenchel and Allendoerfer or Allendoerfer-Weil 
\cite{AllendoerferWeil} to the generalized Gauss-Bonnet-Chern result
using Riemannian polyhedra, where $M$ is triangulated into a simplicial complex with small simplices.
Integral geometry in Allendoerfer-Weil uses tube methods developed of Herman Weyl \cite{Weyl1939}.
Index expectation does neither use tensor calculus, nor tube methods. 

\paragraph{}
It is too bad that $\gamma_d(M)$ is not coordinate independent and that one has to refer to $\gamma(M)$ for that.
In physics, coordinate independence is related to {\bf general covariance}. It is not a coincidence that
Herman Weyl who was so close to physics, started getting interested in topics related to 
invariance theory. Symmetries are pivotal in physics and
the later developments, also in particle physics confirmed this. Also, as Hopf pointed out, a
major motivation for his interest in this topic of curvature in Riemannian geometry 
was due to its relevance to physics. Hopf had been 19 years old, when 
general relativity appeared in 1915. Index expectation is an intuitive tensor-free approach to curvature. If curvature
is given as index expectation, then positive curvature means that more critical points with positive index appear in 
that area than critical points with negative index. This brings index expectation curvature close to the concept of 
``charge". For Morse functions, elementary charges in general either take values $1$ or $-1$. 

\paragraph{}
The sectional Gauss-Bonnet result suggests to look more generally at {\bf $k$-point correlation curvatures} 
of a $2d$-manifold. The {\bf 1-point correlation function} $\langle K_1 \rangle = \sum K_{ij}(x)$ running over
all coordinate planes is the expectation of sectional curvature at a point. 
The integral over $M$ is known as the {\bf Hilbert action}.
The curvature $K$ can be viewed as a {\bf $d$-point correlation curvature} $\langle K_1 K_2...,K_d \rangle$. 
For $d=2$, where we deal with $4$-manifolds, the curvature $K$ is a $2$-point correlation function
$\sum K_1 K_2$, where the sum is over all all coordinate frames. 
In some sense, we can see Gauss-Bonnet-Chern curvature $K_{GBC}(x)$ as
an expectation of the sum of pair-correlations of a fixed plane with various other planes. 

\section{More on the discrete-Continuum relation} 

\paragraph{}
Unexplored is still also whether one can estimate $\delta(M)=\gamma(M)-\chi(M)$ or 
$\delta_d(M)=\gamma_d(M)-\chi(M)$ in cases where 
a group of isometries acts on the Riemannian manifold $M$. Looking for such manifolds with symmetries
is a major line of attack for the Hopf conjecture. Motivations come from
results like the {\bf theorem of Hsiang and Klainer}  \cite{HsiangKleiner}
stating that if an isometric $\mathbb{T}^1$ action exists on a positive curvature $4$ manifold
then $M$ is either homeomorphic (diffeomorphic is only conjectured) 
to the 4-sphere $\mathbb{S}^4$, the projective 4-space $\mathbb{RP}^4$ or
the complex projective plane $\mathbb{CP}^2$. If $M$ is additionally simply connected, then this
reduces $\mathbb{S}^4$ or $\mathbb{CP}^2$. Still keeping simply-connected, but weakening the 
assumption to non-negative curvature adds $S^2 \times S^2$ and a connected sum of $\mathbb{CP}^2 \# \mathbb{CP}^2$. 
A theorem of Berger from 1961 \cite{Berger1961} 
shows that the isometric action on a positive curvature manifold must have a fixed point. 

\paragraph{}
In general, without any positive curvature assumption, there is also a quite classical {\bf theorem of Conner
and Kobayashi} which asserts that the fixed point set $M$ of any smooth circle action (without even needing 
the action to be an isometry with respect to any Riemannian metric on $M$) must be a smooth manifold 
$F$ within $M$ with $\chi(F)=\chi(M)$. Now, this is interesting also from the discrete point of view.
For any isomorphism $T$ of a simplicial complex, there is a Lefshetz fixed point 
theorem \cite{brouwergraph} which states that the Lefshetz number $L(G,T)$ is equal to the sum $\sum_{x \in G} i_T(x)$
of indices. Now, since $L(G,T)$ the super trace of the induced action on cohomology is for the identity $L(G,Id)=\chi(G)$ 
and $i_T(x) = \omega(x)$ for the identity, this reduces to the definition of Euler characteristic in the case of the
identity. 

\paragraph{}
Now, if we think of a manifold as composed of infinitesimal simplices (this is natural from 
the non-standard point of view for example, where simplices have in general much higher dimension than the manifold
but where the simplicial complex constructed from an equivalence relation coming from an infinitesimal distance $h>0$ 
is homotopic to the manifold), the discrete result immediately suggests the corresponding {\bf Lefshetz theorem} in the
continuum. Now, in the case of a circle action $T_t$, the Lefshetz number $L(G,T_t)$ has to be $\chi(G)$
and the sum of the indices of the fixed points $\sum_{x \in G} i_T(x)$ becomes the Euler characteristic of the 
fixed point set. So, the Conner-Kobayashi theorem can be seen as a classical manifestation of the Lefshetz fixed
point theorem for simplicial complexes. 

\paragraph{}
What happens in the case of a positive curvature $2d$-manifold which admits a circular action of isometries?
The only missing brick from the following statement is the issue of even co-dimension, but it would 
allow to {\bf boot-strap the Hopf conjecture} from higher dimensions to a lower dimensional part. 

\quest{
Is the {\bf fixed point manifold} $F$ of an isometric circle action different from the identity 
on an even-dimensional positive curvature manifold an {\bf even dimensional} 
positive curvature manifold? We know it has the same Euler characteristic by  Conner
and Kobayashi.
}

\paragraph{}
A stronger version of that question even would ask whether the fixed point manifold 
$F$ admits a circle action (if it has positive dimension) which would answer the Hopf 
conjecture for positive curvature manifolds admitting an isometric circle action. 
The intuition is that that one could speed up the Killing vector 
field near the fixed point set $F$ so that the vectors remain have length $1$ and then
in the limit produce so a Killing vector field on the fixed point manifold $F$. 
That would then allow induction in dimension in the class of positive curvature manifolds 
admitting a Killing vector field. Since things are settled in dimension $2$ (and $4$) 
that would settle it over all for all dimensions. 

\paragraph{}
A simple example is a rotationally symmetric $2d$-sphere rotated around a line in 
$\mathbb{R}^{2d+1}$, where the fixed point set consists of a $0$-sphere, which are 2 points. 
If the question (even without the strong assumption of having an induced flow on $F$) 
is true then the Hopf conjecture would be settled in particular
for positive curvature 6-manifolds
admitting a circle action, the reason being that all positive curvature 
$4$ and $2$ manifolds are known to have positive Euler characteristic
and for 0-dimensional manifolds positive Euler characteristic follows from 
the fact that the fixed point set can not be empty, by Berger's theorem \cite{Berger1961}.

\bibliographystyle{plain}

\end{document}